\documentclass[twocolumn]{autart}

\usepackage{balance}
\usepackage{graphics}
\usepackage{graphicx}
\usepackage{amssymb}
\usepackage{verbatim}
\usepackage{amsxtra}
\usepackage{epsfig}
\usepackage{subfigure}
\usepackage{hyperref}
\usepackage{amsmath}
\usepackage{latexsym}
\usepackage{ifthen}
\usepackage{psfrag}
\usepackage{color}
\usepackage{rgsEnvironments}
\usepackage{rgsMacros}
\usepackage{pstool}

\usepackage{rgsMacros} 
\usepackage{rgsEnvironments}
\usepackage{DubinsSpecificMacros}
\usepackage[normalem]{ulem}

\newcommand{\cal}{\mathcal}
\newtheorem{problem}{Problem}
\newenvironment{proof}[1]{{\bf Proof}:\ #1}{\null\hfill $\blacksquare$}

\newboolean{Proofs}
\setboolean{Proofs}{true}
\newcommand{\ShowProofs}[1]{\ifthenelse{\boolean{Proofs}}{#1}{}}

\newcommand{\yong}[1]{{\color{black}#1}}

\renewcommand{\ricardo}[1]{{\color{black}#1}}

\usepackage{array}
\usepackage{booktabs}
\renewcommand{\P}{{\cal P}}

\begin{document}

\begin{frontmatter}

\title{On Minimum-time Paths of Bounded Curvature with Position-dependent Constraints}

\author[First]{Ricardo G. Sanfelice}
\author[Second]{Sze Zheng Yong}
\author[Second]{Emilio Frazzoli}

\address[First]{Department of Aerospace and Mechanical Engineering, University of Arizona
1130 N. Mountain Ave, AZ 85721, USA (Tel: 520-626-0676; e-mail: sricardo@u.arizona.edu).}
\address[Second]{Laboratory for Information and Decision Systems, Massachusetts Institute of Technology, Cambridge, MA 02139-4307, USA (e-mail: szyong@mit.edu, frazzoli@mit.edu).}

\begin{abstract}
\noindent
We consider the problem of a particle traveling from an initial configuration to a final configuration (given by a point 
in the plane along with  a prescribed velocity vector) in minimum time with 
\ricardo{non-homogeneous velocity and with constraints} on the minimum turning radius of the particle over multiple regions of the state space.
Necessary conditions for optimality of these paths are derived to characterize the nature of optimal paths,
both when the particle is inside a region and when it crosses boundaries between neighboring regions.
These conditions are used to characterize families of optimal and nonoptimal paths.
Among the optimality conditions, we derive a 
``refraction'' law at the boundary of the regions that generalizes the so-called
Snell's law of refraction in optics to the case of paths with bounded 
curvature. 
Tools employed to deduce our results include recent principles of optimality for hybrid systems. 
The results are validated numerically.
\end{abstract}
\end{frontmatter}

\section{Introduction}
\label{sec:Introduction}

\subsection{Background}
\label{sec:Preliminaries}

Pontryagin's Maximum Principle \cite{Pontryagin62} is a very powerful tool to derive
necessary conditions for optimality of solutions to a dynamical system. 
In other words, this principle
establishes the existence of an adjoint function with the property that, 
along optimal system solutions,
the Hamiltonian obtained by combining the system dynamics and the cost function 
associated to the optimal control problem is minimized.
In its original form, this principle is applicable to optimal control problems with 
dynamics governed by differential equations with continuously differentiable
right-hand sides. 

The shortest path problem between two points with specific tangent directions
and bounded maximum curvature has received wide attention in the literature. 
In his pioneer work in \cite{Dubins57}, by means of geometric
arguments, Dubins showed that optimal paths to this problem 
consist of a concatenation of no more than three pieces, each of them
describing either a straight line, denoted by $\L$, or a circle, denoted by $\C$
(when the circle is traveled clockwise, we label it as $\C^+$,
while when the circle is traveled counter-clockwise, 
$\C^-$),
and are either of type $\C\C\C$ or $\C\L\C$, that is, 
they are among the following six types of paths
\begin{eqnarray}\label{eqn:DubinsPaths}
\C^-\C^+\C^-,
\C^+\C^-\C^+,
\C^-\L\C^-,
\C^+\L\C^+,
\C^+\L\C^-,
\C^-\L\C^+
\end{eqnarray}
in addition to any of the subpaths obtained when some of the pieces (but not all)
have zero length. More recently, the authors in \cite{BoissonnatCerezoLeblond94}
recovered Dubins' result by using Pontryagin's Maximum Principle; see also
\cite{SussmannTang91}. Further investigations of the properties of optimal paths to this problem and 
other related applications of 
Pontryagin's Maximum Principle include
\cite{ShkelLumelsky01,BalkcomMason02IJRR,ChitsazLaValleBalkcom06},
to just list a few.

\subsection{Contributions}

We consider the minimum-time problem of having time-parameterized paths with bounded curvature for a particle, 
which, as in the problem by Dubins, travels from a given initial point to a final point with specified velocity vectors, but
with non-homogeneous traveling speeds and curvature constraints: the velocity of the particle and the 
minimum turning radius are possibly different at certain regions of the state space.
(Note that since the velocity of the particle in the problem by Dubins is constant, the minimum-length and minimum-time problems are equivalent; while
 the problem with different velocities and curvature constraints is most interesting for the minimum-time case.)
Such a heterogeneity arises in several robotic motion planning problems across environments with obstacles, 
different terrains properties, and other topological constraints. Current results for optimal control under heterogeneity,
which include those in \cite{AlexanderRowe90,RoweAlexander00},
are limited to particles describing straight paths.
Furthermore, optimal control problems exhibiting such discontinuous/impulsive behavior
cannot be solved using the classical Pontryagin's Maximum Principle. 
Extensions of this principle to systems with discontinuous right-hand side appeared
in \cite{Sussmann97} while extensions to hybrid systems include
\cite{Sussmann99}, \cite{GaravelloPiccoli05}, and \cite{ShaikhCaines07}.
These principles establish the existence of an adjoint function
which, in addition to conditions that parallel the necessary optimality conditions
in the principle by Pontryagin, satisfies certain conditions at times of discontinuous/jumping behavior.
The applicability of these principles to relevant problems have been highlighted in
\cite{Sussmann99,Piccoli99,DApiceGaravelloManzoPiccoli03IJC}. 
These will be the key tools in deriving the results in this paper.

Building from preliminary results in \cite{Sanfelice.08.HSCC.Dubins} and exploiting 
recent principles of optimality for hybrid systems,
we establish necessary conditions for optimality of paths of particles with bounded curvature traveling across
a state space that is partitioned into multiple regions, each with a different velocity and minimum turning radius. Necessary conditions for optimality of these paths are derived to characterize the nature of optimal paths,
both when the particle is inside a region and when it crosses boundaries between neighboring regions.
A ``refraction'' law at the boundary of the regions that generalizes the so-called
Snell's law of refraction in optics to the case of paths with bounded 
curvature is also derived. The optimal control problem with a ``refraction" law at the boundary can be viewed as an extension of optimal control problems in which the terminal time is governed by a stopping constraint, as considered in \cite{Lin.et.al.11,Lin.et.al.12}. The necessary conditions we derived
also provide a novel alternative to optimizing a switched system without directly optimizing the switching times as decision variables, as is commonly done in a vast majority of papers dealing with switched system optimization, e.g. \cite{Wu.Teo.06,Jiang.Teo.et.al.12}.
Applications of these results include optimal motion planning of autonomous vehicles 
in environments with obstacles, different terrains properties, and other topological constraints.
Strategies that steer autonomous vehicles across heterogeneous terrain using 
Snell's law of refraction have already been recognized in the literature 
and applied to point-mass vehicles; see, e.g.,
\cite{AlexanderRowe90,RoweAlexander00}, and more recently, \cite{KwokMartinez.10.ACC}. 
Our results extend those to the case of autonomous vehicles with 
Dubins dynamics, consider the case when the state space is partitioned into finitely many regions,
and allow for the velocity of travel and minimum turning radius to change in each region.
The results are validated numerically using the software package GPOPS \cite{Rao.10}.

The organization of the paper is as follows.
Section~\ref{sec:ProblemStatement} states the problem of interest and outlines the 
solution approach.  
Section~\ref{sec:Results} presents the main results:
necessary conditions for optimality of paths,
refraction law at the boundary of the regions, and characterization
of families of optimal and nonoptimal paths.
The results are validated numerically in Section~\ref{sec:Example}.

\subsection{Notation}
\label{sec:Notation}

We use the following notation throughout the paper. 
$\reals^{n}$ denotes $n$-dimensional Euclidean space.
$\reals$ denotes the real numbers.
$\realsgeq$ denotes the nonnegative real numbers, i.e.,
$\realsgeq=[0,\infty)$.
$\nats$ denotes the natural numbers including $0$, i.e.,
$\nats=\left\{0,1,\ldots \right\}$. Given $k \in \nats$, $\nats_{\leq k}$ 
denotes $\{0,1,\ldots,k\}$ and, if $k > 0$, 
$\nats_{< k}$ 
denotes $\{0,1,\ldots,k-1\}$.
Given a set $S$, $\ol{S}$ denotes its closure, $S^{\circ}$ denotes its interior,
and $\partial S$ denotes its boundary.
Given a vector $x\in \reals^n$, $|x|$ denotes the Euclidean vector norm.
Given vectors $x$ and $y$,
at times, we write $[x^\top, y^\top]^\top$
with the shorthand notation $(x,y)$. 
\ricardo{Given a function $f$, its domain is denoted by $\dom f$}.
Given $\umax_i >0$ defining $U_i := [-\umax_i,\umax_i]$,
$\U_i$ denotes the set of all piecewise-continuous functions $u$ 
from subsets of $\realsgeq$ to $U_i$.
The inner product between vectors $u$ and $v$ is denoted by
$\langle u, v \rangle$.
A unit vector with angle $\theta$ is denoted by $\angle \theta$.

\section{Problem Statement and Solution Approach}
\label{sec:ProblemStatement}

We are interested in deriving necessary conditions for a path $\X$ describing the motion of 
a particle, which starts and ends at pre-established points with particular velocity vectors,
\yong{through} regions $\P_q$ with different constant velocity $v_q$ and minimum turning radius $r_q$. 
The  dynamics of a particle with 
position $(x,y) \in \reals^2$ and orientation $\theta \in \reals$ (with respect to the vertical axis) are given by
\begin{eqnarray}
\label{eqn:Dubins}
\begin{array}{ccl}
\xdot\ =\ v_q \sin \ang, \ \
\ydot\ =\ v_q \cos \ang, \ \
\angdot\ = \ u,
\end{array}
\end{eqnarray} 
where
$u \in U_q$ is the angular velocity input and satisfies
$|u| \leq \overline{u}_q := \frac{v_q}{r_q}$.
 The velocity vector of the particle is given by the vector
 $\left[
 v_q \sin \ang,\ 
v_q \cos \ang 
 \right]^\top$.
More precisely, we are interested in the following problem:

\begin{problem}
\label{problem:NRegions}
Given a connected set $\P\subset \reals^2$, $N$ disjoint 
polytopes
 $\P_1, \P_2, \ldots, \P_N$, subsets of $\P$, with 
nonempty interior and such that $\P = \cup_{q \in \Pset} \P_q$, determine
necessary conditions on the minimum-time path $\X \subset \P$ of a particle starting at a point $(x,y)^i \in \P_{q^i}^{\circ}$, $q^i \in \Pset$, with initial velocity vector $\nu^i$,
traveling according to \eqref{eqn:Dubins}, and ending at a point
$(x,y)^f \in \P_{q^f}^{\circ}$, $q^f \in \Pset$, with final velocity vector $\nu^f$,
where, for each $q \in \Pset$, $v_q > 0$
and $r_q>0$ are the velocity of travel and 
the minimum turning radius in $\P_q$, respectively.
\fintriangle
\end{problem}

In addition to Problem~\ref{problem:NRegions},
we consider the special case when the angular velocity constraints on neighboring 
regions have common bounds $\overline{u}_q$. We refer
to the resulting problem as
Problem~\ref{problem:NRegions}$\star$.

Figure~\ref{fig:NPatches} depicts the general scenario in Problem~\ref{problem:NRegions}.
Neighboring regions are such that either their velocity of travel, their minimum turning
radius, or both are different from each other.
In this way, the number of regions with different characteristics is irreducible.

\begin{figure}[h!]  
  \begin{center}  
     \psfrag{M}[][][0.9]{$\boundary$}
     \psfrag{zi}[][][0.8]{$(x,y)^i$}
     \psfrag{vi}[][][0.8]{$\nu^i$}
     \psfrag{r1}[][][0.8]{$r_1$}
     \psfrag{r2}[][][0.8]{$r_2$}
     \psfrag{r3}[][][0.8]{$r_3$}
     \psfrag{v1}[][][0.8]{$v_1$}
     \psfrag{v2}[][][0.8]{$v_2$}
     \psfrag{v3}[][][0.8]{$v_3$}
     \psfrag{zf}[][][0.8]{\hspace{0.1in}$(x,y)^f$}
     \psfrag{vf}[][][0.8]{$\nu^f$}
     \psfrag{1}[][][0.8]{${\cal P}_1$}
     \psfrag{2}[][][0.8]{${\cal P}_2$}
     \psfrag{3}[][][0.8]{${\cal P}_3$}
     \psfrag{y=0}[][][0.8]{}
     \psfrag{t0}[][][0.8]{}
     \psfrag{t1}[][][0.8]{}
    {\includegraphics[width=.21\textwidth]{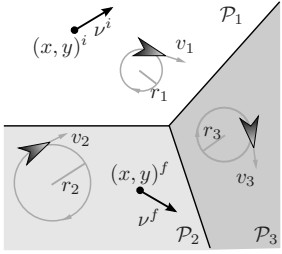}}  
  \end{center}  
\caption{Initial point $(x,y)^i$ and final point $(x,y)^f$ with given velocity vectors on regions ${\cal P}_1$ and ${\cal P}_2$. The minimum turning radius in region $\P_1$ is smaller than the one in region $\P_3$, which is smaller than the one in region $\P_2$ as denoted by the depicted paths with minimum turning radius. \label{fig:NPatches}}
\end{figure}

Our approach to derive a solution to Problem~\ref{problem:NRegions} is as follows.
Given a continuously differentiable curve $\X \subset \P$ defining the path of a particle starting at a point $(x,y)^i \in \P^{\circ}_{\ricardo{q^i}}$, $\ricardo{q^i} \in \Pset$, 
with initial velocity vector $\nu^i$ 
and ending at a point $(x,y)^f \in \P^{\circ}_{\ricardo{q^f}}$, $\ricardo{q^f} \in \Pset$, with final velocity vector $\nu^f$,
we partition the curve $\X$ into a finite number of curves $\X_j$, $j \in \{0,1,\ldots,J-1\} (= \naturals_{< J})$, 
such that, for each $j \in \naturals_{< J}$, $\X_j \in \P_{q_j}$ for some $q_j \in \Pset$,
where, for each $j \in \naturals_{< J}$, $q_{j} \not = q_{j+1}$.
This partition is such that the velocity of the particle along each curve $\X_j$ is constant
and given by $v_{q_j}$ while the minimum turning radius is also constant and given by $r_{q_j}$.
For each $j \in \naturals_{< J}$, let $t_j\geq 0$ denote the time at which the
particle starts the path defined by the curve $\X_j$ and let $t_J$ denote the 
time at which the particle reaches the \yong{end} point $(x,y)^f$.

\ricardo{To formally define trajectories to \eqref{eqn:Dubins} that are solutions to Problem~\ref{problem:NRegions},
we conveniently}
parametrize $\x$, $\y$, $\ang$, \ricardo{and $q$} by functions defined for each $j \in \naturals_{< J}$ on $[t_j,t_{j+1}]\times\{j\}$, 
where, if necessary, the value at the boundaries (left or right) of the interval $[t_j,t_{j+1}]$ are obtained by taking 
(right or left, respectively) limit of the functions. 
\ricardo{The second argument $j$ specifies the number of 
crossings between regions.  With this parametrization, the 
current region of the particle
at time $t$ and after $j$ crossings
is given by the value of $q$ at $(t,j)$, that is, $q(t,j)$.}
Hence, the domain of $\x$, $\y$, $\ang$, \ricardo{ and $q$} is given by the following subset of $\realsgeq \times \naturals$:
\begin{eqnarray}\label{eqn:htd}
E:=\bigcup_{j \in \naturals_{< J} } \left( [t_j,t_{j+1}]\times\{j\} \right)
\end{eqnarray}
on which, for each $j \in \naturals_{< J}$, 
the function $(t,j) \mapsto [\x(t,j),\  \y(t,j),\ \ang(t,j)\ricardo{, \ q(t,j)}]^\top$ is absolutely continuous and
for each $(t_j,j) \in E$, $j \in \{1,2,\ldots,J-1\}$, $[\x(t_{\ricardo{j+1}},j+1),\  \y(t_{\ricardo{j+1}},j+1),\ \ang(t_\ricardo{j+1},j+1)]^\top = [\x(t_{\ricardo{j+1}},j),\  \y(t_{\ricardo{j+1}},j),\ \ang(t_{\ricardo{j+1}},j)]^\top$, i.e.,
the position and angle of the vehicle do not jump when changing regions, 
\ricardo{while $q(t_{j+1},j+1)$ is equal to the new region's index.}\footnote{This particular construction of time domain is borrowed from \cite{Hybrid01}, where it is called a (compact) {\em hybrid time domain}. }

With this re-parametrization of the path $\X$, we employ the principle of optimality for hybrid systems
in \cite{Sussmann99} (see also \cite{Sussmann99LNCIS} and \cite{Piccoli99})
to derive necessary conditions for the curve $\X$ to be optimal. To this end, we 
associate the re-parametrized curve $\X$ with a solution to a hybrid system, which via results
in \cite{Sussmann99}, infer conditions for optimality. More precisely, $[\x,\ \y,\ \ang]^\top$ along with 
a function $q$ with domain \eqref{eqn:htd}, which takes discrete values $1,2,\ldots,N$ 
when the velocity of the particle is equal to $v_1,v_2,\ldots,v_N$ and the minimum turning
radius is equal to $r_1,r_2,\ldots,r_N$, respectively, define a solution 
to a hybrid system.  The continuous evolution
of $[\x,\ \y,\ \ang]^\top$
is given
as in \eqref{eqn:Dubins}
with
$q$ being a function that is constant during periods of flow, 
but when the particle crosses from one region to another, $q$ is updated to indicate 
the region associated with the current velocity and minimum turning radius.
The angular velocity input
$u$
is designed to satisfy the constraint $|u|\leq \overline{u}_{q}$ 
for the value of $q$ corresponding to the current region of travel.

The following assumptions are imposed on some of the results to follow.
For notational simplicity in the following sections, we define $\zcont := [\x,\ \y,\ \ang]^{\top}$ and $Q:=\Pset$.
\begin{assumption}
\label{assump:OptimalTrajectories}
Given a solution $(\zcont,q)$ to Problem~\ref{problem:NRegions}, 
the following holds:
\begin{enumerate}
\item\label{assump:OptimalTrajectories1}
The set of points $\defset{t}{ (t,j) \in \dom (\xi,q), (x(t,j),\right.$ \\$\left.y(t,j)) \in \partial \P_{q(t,j)}}$ 
has zero Lebesgue measure.
\item\label{assump:OptimalTrajectories2}
For every $(t,j)$ such that $(t,j+1) \in \dom (\xi,q)$,
there exist unique $p, p' \in Q$ such that $(x,y)(t,j) \in \partial \overline{\P}_{p} \cap \partial \overline{\P}_{p'}$
and the boundaries $\partial \overline{\P}_{p}$ and $\partial \overline{\P}_{p'}$
are locally smooth at $(x(t,j),y(t,j))$.
\end{enumerate}
\end{assumption}
A solution $(\zcont,q)$ to Problem~\ref{problem:NRegions}
satisfying condition \ref{assump:OptimalTrajectories1} in Assumption~\ref{assump:OptimalTrajectories}
does not slide on the boundary of the regions
while, if it satisfies condition \ref{assump:OptimalTrajectories2},
then it crosses 
at points where only two regions exist with 
their boundaries defined by a smooth curve.

\section{Properties of Optimal Paths}
\label{sec:Results}

\subsection{Necessary conditions for optimality}
\label{sec:OptimalControl}

The following lemma establishes basic conditions that the state variables $\zcont$
and associated adjoint satisfy for the case of two regions $\P_1$ and $\P_2$.
It employs \cite[Theorem 1]{Sussmann99}, which, 
under further technical assumptions, 
establishes that
there exists an {\em adjoint pair} $(\adj,\adj_{\circ})$, where 
$\adj$ is a function
and $\adj_{\circ}$ is a constant, which, along optimal solutions to
Problem~\ref{problem:NRegions},
 satisfies certain {\em Hamiltonian maximization}, {\em nontriviality}, 
{\em transversality}, and {\em Hamiltonian value conditions}. Note that the restriction that the boundary of the switching surfaces is a subset of \ricardo{$\defset{(x,y) \in \reals^2}{ y = 0 }$} in the following lemma will be relaxed in the ensuing discussion.

\begin{lemma}
\label{propo:PropertiesTwoRegions}
For each optimal solution $(\zcont,q)$ to 
 Problem~\ref{problem:NRegions} 
with optimal control $u$, minimum transfer time $T$, number of jumps $J-1$,
number of regions $N = 2$, and 
($\overline{\P}_{1} \cap \overline{\P}_{2})\cap \P \subset \defset{(x,y) \in \reals^2}{ y = 0 }$
satisfying Assumption~\ref{assump:OptimalTrajectories}\footnote{\yong{An optimal solution $(\xi,q)$ uniquely determines $T$ and $J$.}},
there exists a function
$\adj:\dom \adj \to \ODubins$, $\adj:= [\adjx,\ \adjy,\ \adjang]^{\top}$, $\dom \adj = \dom \HSvsol$,
where $t \mapsto \adj(t,j)$ is absolutely continuous for each $j$, $(t,j)\in \dom \adj$,
and a constant $\adj_{\circ}\in \reals$
defining the adjoint pair $(\adj,\adj_{\circ})$
satisfying
\begin{itemize}
\item[a)] $\dot{\adj}(t,j) = -\displaystyle\frac{\partial H_{q(t,j)}}{\partial \zcont}(\zcont(t,j),\adj(t,j),\adj_{\circ},u(t,j))$ and $\adj_{\circ}\geq 0$
for almost every $t \in [t_j,t_{j+1}]$, $(t,j) \in \dom \adj$,
where, for each $q\in Q$, 
$H_q$
is the Hamiltonian associated with the continuous dynamics of the hybrid system $\HS$, which
is given by
$$H_q(\zcont,\adj,\adj_{\circ},u)  = \adjx v_{q} \sin \ang+\adjy v_{q} \cos \ang + \adjang u - \adj_{\circ}.$$
\item[b)] There exist $\overline{\adjx},\overline{\adjy} \in \reals$ and, for each $j \in \nats_{\leq J}$, 
there exists $\p_j \in \reals$
such that
$\adjx(t,j):= \overline{\adjx}$ for all $(t,j) \in \dom \HSvsol$, $\adjy(t,j):= \overline{\adjy}+p_j$ for almost all $t \in [0,T]$, $(t,j) \in \dom \HSvsol$,
and $\adjang(t,j)=\adjang(t,j+1)$ 
for each $(t,j)$ such that $(t,j),(t,j+1) \in \dom \adj$.
\item[c)] The control $\u$ is equal to $\umax_{q(t,j)}$
when $\gamma(t,j)>0$, $-\umax_{q(t,j)}$
when $\gamma(t,j)<0$, and $0$
when $\gamma(t,j)=0$.
\item[d)] 
For every $j$ such that there exists an interval
$I_j$ with nonempty interior, $I_j \times \{j\} \subset \dom \HSvsol$, such that
 $\adjang(t,j) = 0$ for each $t \in I_j$, then $\adjy(t,j) \tan \ang(t,j) = \adjx(t,j)$ for each $t \in \ol{I_j}$.
 \item[e)] There exists $c \in \reals$ such that for every $(t,j) \in \dom \HSvsol$
\begin{eqnarray}\label{eqn:HamiltonianCondition}
\begin{array}{ll}
H_{q(t,j)}(\zcont(t,j),\adj(t,j),\adj_{\circ},\u(t,j)) = c\ .
\end{array}
\end{eqnarray}
\end{itemize}
\end{lemma}
\begin{proof}
For each $q \in Q:=\{1,2\}$, let $\M_q := \ODubins$, $J_q := \reals$, 
and recall the definition of $U_q = [-\overline{u}_q,\overline{u}_q]$ from Section~\ref{sec:Notation},
where $\U_q$ is
the set of all piecewise-continuous functions from subsets of $\realsgeq$ to $U_q$.
Let ${\cal S}$ be a subset of $\hat{\M}^2:=\left(\cup_{q \in Q}\left(\{q\}\times\M_q\right)\right) \times \left(\cup_{q \in Q}\left(\{q\}\times\M_q\right)\right)$ given by
\begin{eqnarray}
{\cal S} := \defset{(q,\zcont,q',\zcont)}{\y=0, (q,\zcont,q',\zcont)\in\hat{\M}^2}
\end{eqnarray}
and, for each $q, q' \in Q$, let
\begin{eqnarray}\label{eqn:SwitchingSet}
{\cal S}_{q,q'} := \defset{(\zcont,\zcont)}{\y=0, \zcont\in\ODubins}\ .
\end{eqnarray}
Then, the above definitions determine a hybrid system following the framework in \cite{Sussmann99}. 
We denote this hybrid system by $\HS$. Its continuous dynamics are given by
$\dot{\zcont}=f_q(\zcont,u)$, where $f_q(\zcont,u):=[v_q\cos\ang,\ \ v_q\sin\ang,\ \ u]^{\top}$ and discrete dynamics given by the switching sets ${\cal S}_{q,q'}$.
By construction, the functions $\x,\y,\ang,$ define a solution $\HSvsol$ to $\HS$ on $E$ with input $u$. 
This follows from the fact that 
jumps of $\HSvsol$ occur at different times $t_j$, $j \in \{1,2,\ldots,J-1\}$, with $t_j \in (0,t_J)$;
$\HSvsol$ is absolutely continuous on each $[t_j,t_{j+1}]\times\{j\}$ for
each $j \in \{0,1,\ldots,J-1\}$;
$\zcont$ and $u$ satisfy the differential equation $\dot{\zcont}=f_q(\zcont,u)$ for almost all $[t_j,t_{j+1}]\times\{j\}$ for
each $j \in \{0,1,\ldots,J-1\}$;
$q$ is constant during flows; 
and $\HSvsol$ satisfies the switching condition 
\begin{eqnarray}\label{eqn:SwitchingOfSolution}
\begin{array}{ll}
(q(t_j,j),\zcont(t_j,j),q(t_j,j+1),&\zcont(t_j,j+1))\\& \in   {\cal S}_{q(t_j,j),q(t_j,j+1)} 
\end{array}
\end{eqnarray}
for each $(t_j,j) \in E$, $j>0$.
Then, for each $j \in \{0,1,\ldots,J-1\}$, the functions
$\zcont_{j+1}:[t_{j},t_{j+1}]\to\ODubins$, $q_{j+1}:[t_{j},t_{j+1}]\to Q$, $u_{j+1}:[t_{j},t_{j+1}]\to U_{q_{j+1}}$ 
given by $\zcont_{j+1}(t) := \zcont(t,j)$, $q_{j+1}(t) := q(t,j)$, and $u_{j+1}(t) := u(t,j)$
for each $(t,j) \in E$ define a solution $(\zcont_1, \zcont_2, \ldots, \zcont_{J}), (q_1, q_2, \ldots, q_{J})$ 
to $\HS$ as in \cite[Definition 3]{Sussmann99} with control input $(u_1,u_2,\ldots,u_{J})$.
Then, for $\HSvsol$ and $u$ to be optimal, the solution and input need to satisfy the necessary conditions for optimality in 
\cite[Theorem 1]{Sussmann99}. These establish items a)-d) as shown below. 

By \cite[Theorem 1]{Sussmann99}, there exist a piecewise absolutely continuous function
$\adj:\dom \adj \to \ODubins$, $\adj:= [\adjx,\ \adjy,\ \adjang]^{\top}$, with domain $E$
and a constant $\adj_{\circ}\in \reals$ defining the adjoint pair $(\adj,\adj_{\circ})$
satisfying
\begin{eqnarray}\label{eqn:AdjointCondition}
\adj_{\circ} \geq 0,  \dot{\adj}(t,j) = -\frac{\partial H_{q(t,j)}}{\partial \zcont}(\zcont(t,j),\adj(t,j),\adj_{\circ},u(t,j))
\end{eqnarray}
for almost every $t \in [t_j,t_{j+1}]$, $(t,j) \in E$, where, for each $q$,
$H_q$ is given by
\begin{eqnarray}
H_q(\zcont,\adj,\adj_{\circ},u) := \langle \adj,f_q(\zcont,u)\rangle - \adj_{\circ}L_q(\zcont,u)
\end{eqnarray}
with $L_q(\zcont,u) \equiv 1$.
This establishes item a). 

From \eqref{eqn:AdjointCondition}, it follows that $\adj$ satisfies
\begin{eqnarray}\label{eqn:AdjointDynamics}
\dot{\adjx} = 0, \qquad \dot{\adjy} = 0, \qquad \dot{\adjang} = -\adjx v_q \cos \ang + \adjy v_q \sin \ang\ .
\end{eqnarray}
Then, $\adjx$ and $\adjy$ are piecewise constant. 
By \cite[Theorem 1]{Sussmann99}, for each $(t_j,j) \in E$, $j \in \{1,2,\ldots,J-1\}$, $\adj$
satisfies
\begin{eqnarray}\label{eqn:JumpConditionCone}
(-\adj(t_j,j),\adj(t_j,j+1)) \in K^{\perp}_{j}\ ,
\end{eqnarray}
where $K^{\perp}_j$ is the polar of the Boltyanskii approximating cone to ${\cal S}_{q(t,j),q(t,j+1)}$
at the jump at $(t,j)$.
\footnote{Given a subset $S$ of a smooth manifold $\M$, the Boltyanskii approximating cone to $S$ at a point
$x \in \M$ is a closed convex cone $K$ in the tangent space $T_x \M$ to $\M$ at $x$ such that there exists a 
neighborhood $V$ of $0$ in $T_x \M$ and a continuous map $\mu:V\cap K \to \M$ with the property that $\mu(V\cap K)\subset S$, 
$\mu(0)= x$, $\mu(v) = x + v + O(|v|)$ as $v \to 0$, $v \in V\cap K$. See, e.g., \cite[Definition 8]{Sussmann99}.}
Let $\hat{\cal S}:={\cal S}_{1,2}={\cal S}_{2,1}$. 
The Boltyanskii approximating cone to the set $\hat{\cal S}$ is the set itself. Then, for each
$j \in \{1,2,\ldots,J-1\}$, $K^{\perp}_j$ is given by \footnote{Given $z,z' \in \reals^m\times\reals^n$, $\langle z,z' \rangle$ follows the inner product definition in Section~\ref{sec:Notation} and is the inner product between the vectors obtained from stacking the columns of $z$ and $z'$.}
$$K^{\perp} = \defset{w \in \ODubins\times\ODubins}{\langle w,v \rangle \leq 0\ \forall v \in \hat{\cal S}}\ .$$ 
Note that $w \in K^{\perp}$ if and only if
$\langle w,v \rangle \leq 0$ for all $v:= ([v^a_1,\ v^a_2,\ v^a_3]^{\top},[v^b_1,\ v^b_2,\ v^b_3]^{\top}) \in \ODubins\times\ODubins$ such that $v^a_2=v^b_2= 0$.
Hence, condition \eqref{eqn:JumpConditionCone} implies that, for each $(t_j,j) \in E$, $j \in \{1,2,\ldots,J-1\}$,
\begin{eqnarray}\non
\langle -\adj(t_j,j)+\adj(t_j,j+1),v'\rangle=0
\end{eqnarray}
for all $v':= [v'_1,\ v'_2,\ v'_3]^{\top} \in \ODubins$ such that  $v'_2 = 0$.
Then, for each such $(t_j,j)$,
$\adjx(t_j,j)=\adjx(t_j,j+1)$, $\adjang(t_j,j) = \adjang(t_j,j+1)$, 
and only $\adjy$ can jump. This shows item b).

Claim c) follows directly from the Hamiltonian maximization condition
guaranteed by \cite[Theorem 1]{Sussmann99}.
\cite[Theorem 1]{Sussmann99} 
establishes that $u$ satisfies
\begin{eqnarray}\label{eqn:MaximizationofH}
\begin{array}{ll}
H_{q(t,j)}&(\zcont(t,j),\adj(t,j),\adj_{\circ},u(t,j)) \\ &= 
\max_{u' \in U} H_{q(t,j)}(\zcont(t,j),\adj(t,j),\adj_{\circ},u')
\end{array}
\end{eqnarray}
for almost every $t \in [t_j,t_{j+1}]$, $(t,j) \in E$.
Then, since $H_q(\zcont,\adj,\adj_{\circ},u) = \adjx v_{q} \sin \ang+\adjy v_{q} \cos \ang + \adjang u - \adj_{\circ}$, 
for each $(t,j)\in E$ for which $\adjang(t,j) \not = 0$, $u(t,j) = \sign(\adjang(t,j))\overline{u}_q \in \{-\overline{u}_q,\overline{u}_q\}$.
When $\adjang(t,j)=0$, $H_q(\zcont,\adj,\adj_{\circ},u)$ does not depend on $u(t,j)$ and hence, the optimum value of $u(t,j)$ is ambiguous. If the ambiguity exists over a time interval, we have the singular arc case. In this case, differentiating $\adjang(t,j)$ twice and setting it to zero yields $\ddot{\adjang}(t,j)=(\adjx v_q \sin \ang +\adjy v_q \cos \ang)u=0$, which leads to $u(t,j) = 0$ if $\adjang(t,j) = 0$. This implies item c). Note that in optimal control, this is referred to as bang-singular control.

To establish item d), note that when $\adjang(t,j)=0$ on $I_j \times \{j\}$, 
since $t \mapsto \adjang(t,j)$ is absolutely continuous, 
we have that $\frac{d}{dt}\adjang(t,j)=0$ for each $(t,j) \in I_j^\circ \times \{j\}$.
Then, item d) follows from \eqref{eqn:AdjointDynamics} and absolute continuity 
of $t \mapsto \ang(t,j)$.

Finally, to show item e) we use the Hamiltonian value condition guaranteed by \cite[Theorem 1]{Sussmann99}.
In fact, since the jump condition in $\HS$ is time independent, that is, $J_1=J_2=\reals$, 
the Hamiltonian value condition in \cite[Theorem 1]{Sussmann99} establishes that 
there exists $c \in \reals$ such that \eqref{eqn:HamiltonianCondition}
holds
for almost every $t \in [t_j,t_{j+1}]$, $(t,j) \in \dom E$.
\end{proof}

Item a) determines the dynamics of the adjoint function $\adj$ while items b), c), and d) establish
conditions that the components of $\adj$, $\zcont$, and $u$ satisfy along the path.
In particular, item c) indicates that the paths either describe a straight line or an arc with minimum turning radius, i.e., of radius $r_p$, 
where $p$ is equal to the corresponding entry in $\{1,2\}$.

The properties highlighted in Lemma~\ref{propo:PropertiesTwoRegions}
can be used to establish conditions for the functions $\x,\y,\ang,$ and $u$ 
in Problem~\ref{problem:NRegions}, since at every crossing of the boundary between
two arbitrary regions, 
a change of coordinates can be performed so that 
the geometry in Lemma~\ref{propo:PropertiesTwoRegions} holds
for the chosen boundary crossing.
 In fact, at every crossing of the boundary between
two arbitrary regions occurring at a position $(x^*,y^*)$,
a change to a coordinate system with vertical axis perpendicular to the tangent
to the boundary of the two regions and vertical coordinate equal to zero when $y=y^*$
can be defined via a rotation plus translation of the original coordinate
system $(x,y,\theta)$. 
More precisely, a coordinate transformation performing this operation is given by
\begin{align} \label{eqn:coordTrans}
\begin{bmatrix} x' \\ y' \\ \theta'  \end{bmatrix}=\begin{bmatrix} \cos \varphi & \sin \varphi & 0 \\ -\sin \varphi & \cos \varphi & 0  \\ 0 & 0 & 1 \end{bmatrix} \begin{bmatrix} x \\ y \\ \theta \end{bmatrix}-\begin{bmatrix} x^* \cos \varphi+y^* \sin \varphi  \\ -x^* \sin \varphi+y^* \cos \varphi \\ -\varphi \end{bmatrix}
\end{align}
where 
$\varphi$ is the angle of the tangent line to the boundary
of the regions where the crossing occurs.
In this setting, 
as in Lemma~\ref{propo:PropertiesTwoRegions},
the optimal paths from given initial and terminal constraints can 
be characterized using the principle of optimality for hybrid systems
in \cite{Sussmann99}.

\begin{theorem}{(optimality conditions of paths to Problem~\ref{problem:NRegions})}
\label{thm:optimalityConditionsPv}
Let the curve $\X$ describe a minimum-time 
path that solves Problem~\ref{problem:NRegions} and let $\x, \y$ and  $\ang$ 
be its associated functions with input $u$.
Define the function $q$ 
following the 
construction below  
 \eqref{eqn:SwitchingSet}.
 Suppose Assumption~\ref{assump:OptimalTrajectories} holds.
 Then, the following properties hold:
\begin{itemize}
\item[a)] The curve $\X$ is a smooth concatenation of finitely many pieces from the set $\DubinsPaths$.
\item[b)] The angular  
velocity input
$u(t,j)$ 
 is piecewise constant with finitely many pieces
taking value in $\{-\overline{u}_{q(t,j)},0,\overline{u}_{q(t,j)}\}
=\{-\frac{v_{q(t,j)}}{r_{q(t,j)}},0,\frac{v_{q(t,j)}}{r_{q(t,j)}}\}$.
\item[c)] For each $j \in \{0,1,\ldots,J-1\}$, each piece $\X_j$ of the curve $\X$
is Dubins optimal between the first and last point of such \yong{a} piece, i.e., it is given as in \eqref{eqn:DubinsPaths}.
\item[d)] For each $(t_j,j) \in E$, $j \in \{1,2,\ldots,J-1\}$, if the last path piece of $\X_{j-1}$ and the first path piece of $\X_j$ are of type $\L$, and moreover, if $v_{q(t_j,j-1)} \not= v_{q(t_j,j)}$,
then $\ang(t_j,j)$ is zero or any multiple of $\pi$.\end{itemize}
\end{theorem}

\ShowProofs{
\begin{proof}
We apply Lemma~\ref{propo:PropertiesTwoRegions} to the subpieces of the curve $\X$,
 one at a time,
after performing the coordinate transformation in \eqref{eqn:coordTrans} if needed.
Consider the $j$-th piece $\X_j$ of the curve $\X$. 
By item c) in Lemma~\ref{propo:PropertiesTwoRegions}, for each $(t,j) \in [t_j,t_{j+1}]\times\{j\}$,
$u(t,j)$  
takes value $-\overline{u}_{q},0$, or $\overline{u}_{q}$. 
When $u(t,j)=\pm\overline{u}_q$, 
the path at such $(t,j)$'s is of type $\C$ (either $\C^+$ or $\C^-$),
while when $u(t,j) = 0$,
the path at such $(t,j)$'s is of type $\L$.
Then, the curve $\X_j$ is a concatenation of paths of type $\{\C^+,\C^-,\L\}$.
Bellman's principle of optimality implies that for $\X$ to be optimal, each piece $\X_j$ 
also has to be optimal. Then, by the original result by Dubins in \cite{Dubins57},
the concatenation of paths that define $\X_j$ is finite.
Proceeding in this way for each piece of $\X$, items a)-c) hold true.
To show item d), we apply Lemma~\ref{propo:PropertiesTwoRegions}.c), d), and e) 
to each pair of pieces $\X_{j-1},\X_j$ for each $j \in \{1,2,\ldots,J-1\}$. 
Using the coordinate transformation \eqref{eqn:coordTrans} at
the boundary of the $(j-1)$-th region and the $j$-th region at which the pair of
pieces $\X_{j-1},\X_j$ connect, we have that $\X_{j-1},\X_j$ connect at a point 
in the set $\defset{(\x,\y)\in\reals^2}{\y=0}$.
Then, the pair of pieces $\X_{j-1},\X_j$ can be treated as the case of two regions
considered in Lemma~\ref{propo:PropertiesTwoRegions}.
Item b) 
in Lemma~\ref{propo:PropertiesTwoRegions}
indicates that 
the adjoint function component $\adjang$ remains constant at jumps, i.e., $\adjang(t_j,j-1) = \adjang(t_j,j)$, which by the arguments above is equal to zero
for $\L$-type paths. Since,
by Lemma~\ref{propo:PropertiesTwoRegions}.c)
the angular velocity input does not change when $\adjang$ is constant, then,
at $(t_j,j-1)$, we have that $u(t_j,j-1) = u(t_j,j)$.
It follows that $\ang(t_j,j-1) = \ang(t_j,j)$.
By Lemma~\ref{propo:PropertiesTwoRegions}.b), $\adjx(t_j,j) = \adjx(t_j,j+1)=\overline{\adjx}$.
Then, the Hamiltonian value condition 
guaranteed by Lemma~\ref{propo:PropertiesTwoRegions}.e)
becomes
\begin{eqnarray}\label{eqn:HamiltonianContinuitynoInput}
\begin{array}{ll}
&\overline{\adjx} v_{q} \sin \ang(t_j,j)+\adjy(t_j,j) v_{q} \cos \ang(t_j,j)
\\&\hspace{-0.2cm} =\overline{\adjx} v_{q^+} \sin \ang(t_j,j+1)+\adjy(t_j,j+1) v_{q^+} \cos \ang(t_j,j+1)
\end{array}
\end{eqnarray}
where we have used the shorthand notation 
$q = q(t_j,j)$ and $q^+ = q(t_j,j+1)$.
Item d) in Lemma~\ref{propo:PropertiesTwoRegions} implies that
$\adjy(t_j,j)\tan \ang(t_j,j)=\overline{\adjx}$
and 
$\adjy(t_j,j+1)\tan \ang(t_j,j)=\overline{\adjx}$.
Substituting these expressions into
\eqref{eqn:HamiltonianContinuitynoInput},
it follows that
\begin{eqnarray}\label{eqn:HamiltonianContinuitynoInput2}
v_{q(t_j,j-1)} \sin\ang(t_j,j) = v_{q(t_j,j)} \sin \ang(t_j,j)\ .
\end{eqnarray}
Since $v_{q(t_j,j-1)} \not= v_{q(t_j,j)}$,
\eqref{eqn:HamiltonianContinuitynoInput2} holds if and only if
$\ang(t_j,j) = 0$ or any multiple of $\pi$.
Repeating this procedure for each piece concludes the proof.
\end{proof}
}

\subsection{Refraction law at boundary}

The necessary conditions of the Hamiltonian $H_q$ given in Lemma~\ref{propo:PropertiesTwoRegions}
relate the values of the state vector $\zcont$, the angular velocity input $u$ and the adjoint vector $\adj$ before and after crossing a boundary between regions. Using 
the properties of the adjoint vector and its relationship to the state vector $\zcont$
in items b)-d) in Lemma~\ref{propo:PropertiesTwoRegions}, it can be shown that an algebraic
condition involving $\adj$, $\zcont$, and certain angles holds at the crossings points for optimal paths of type $\L_p\C_{p-p'}\L_{p'}$
in Problem~\ref{problem:NRegions}. \footnote{To apply Lemma~\ref{propo:PropertiesTwoRegions} 
to the paths
of Problem~\ref{problem:NRegions}, we proceed as in the proof of Theorem~\ref{thm:optimalityConditionsPv}.}
More precisely, the condition on the Hamiltonian in item e)
of Lemma~\ref{propo:PropertiesTwoRegions} 
implies that for each $(t_j,j) \in E$
\begin{align}\label{eqn:HamiltonianJump2}
 &\adjx v_{p} \sin \ang+\adjy v_{p} \cos \ang +  \adjang u  \\
\nonumber &= \adjx^+ v_{p'} \sin(\ang^+)+\adjy^+ v_{p'} \cos(\ang^+) +  \adjang^+ u^+ = c+\adj_{\circ} \ ,
\end{align}
where 
$p = q(t_j,j)$, $p' = q(t_j,j+1)$ and $c ,\adj_{\circ}\in \reals$ are constants.
 For notational convenience, we dropped the $(t_j,j)$'s and denoted the valuation of the functions
at $(t_j,j+1)$ with $\null^+$. 
Moreover, for the path pieces of type $\L$,  items b), c) and e) in Lemma~\ref{propo:PropertiesTwoRegions} imply
\begin{align}\label{eqn:HamiltonianJump3}
\begin{array}{lll}
&\adjx v_{p} \sin \ang_p+\adjy v_{p} \cos \ang_p  
\\ &=
\adjx^+ v_{p'} \sin\ang_{p'}+\adjy^+ v_{p'} \cos \ang_{p'}   = c +\adj_{\circ}\ ,
\end{array}
\end{align}
where $\ang_{p}$ and $\ang_{p'}$ denote
the initial and final angle, respectively, of a path piece intersecting the boundary between $\P_q$ and $\P_{q^+}$, as shown in Figure~\ref{fig:Refraction}.
The algebraic conditions in \eqref{eqn:HamiltonianJump2} and \eqref{eqn:HamiltonianJump3}
can be reduced to a refraction law as established in the following result.

\begin{theorem}{(refraction law)}
\label{thm:refractionPv}
Let the curve $\X$ describe a minimum-length curve that solves Problem~\ref{problem:NRegions} and let $\x, \y$ and $\ang,$
be its associated functions with input $u$. Consider partition pairs $\X_{j-1},\X_j$ for all $j\in\{1,2,\hdots,J-1\}$ such that the path pieces across the boundary are of type $\L_p\C_{p-p'}\L_{p'}$. For each of such pair of pieces, 
suppose the end points (opposite to the intersection with the boundary) of those path pieces are in $\P_p^{\circ}$ and $\P_{p'}^{\circ}$ for some $p, p' \in Q$, respectively,  and suppose Assumption~\ref{assump:OptimalTrajectories} holds.
Let $\Delta \ang_p, \Delta \ang_{p'} \in \reals$
be given by $\Delta \ang_p := \ang^* - \ang_p$,
$\Delta \ang_{p'} := \ang_{p'} - \ang^*$, where
$\ang^*$ 
is the orientation of the vehicle at the boundary, 
i.e.,
it is the angle between the path and the boundary between $\P_p$ and $\P_{p'}$ at their intersection 
with respect to the normal to the boundary.
If the path piece intersecting the boundary between $\P_p$ and $\P_{p'}$ is one of the type $\C_{p-p'}$ paths shown in Figure~\ref{fig:Refraction}, 
then
$v_p, v_{p'},r_p, r_{p'}, \ang_p, \ang_{p'}, \Delta \ang_{p}$ and
$\Delta \ang_{p'}$ satisfy
\begin{eqnarray}\label{eqn:RefractionHSv}
\frac{v_p}{v_{p'}}&=&\frac{\sin \ang_{p} }{\sin \ang_{p'} } \ ,\\
\label{eqn:RefractionHSv2} \frac{r_p}{r_{p'}}&=&\frac{v_p}{v_{p'}}\left(\frac{1-\cos \Delta \ang_{p'}}{1-\cos \Delta \ang_{p}}\right)=\frac{\sin \ang_{p}(1-\cos \Delta \ang_{p'})}{\sin \ang_{p'}(1-\cos \Delta \ang_{p})}\ .
\end{eqnarray}

Moreover, if the path piece intersecting the boundary between $\P_p$ and $\P_{p'}$ is of type $\L$ and $v_p \not = v_{p'}$, then
$\ang_p$ and $\ang_{p'}$ are equal to zero or any multiple of $\pi$.
\end{theorem}
\begin{figure}[ht] 
     \psfrag{Cp}[][][0.9]{$\C^+$}
     \psfrag{Cm}[][][0.9]{$\C^-$}
     \psfrag{L}[][][0.9]{$\L$}
     \psfrag{M}[][][0.9]{$\hspace{0.45in}\boundary$}
     \psfrag{1}[][][0.8]{${\cal P}_p$}
     \psfrag{2}[][][0.8]{${\cal P}_{p'}$}
     \psfrag{R1}[][][0.8]{$r_p$}
     \psfrag{R2}[][][0.8]{$r_{p'}$}
     \psfrag{ang1}[][][0.8]{$\ang_p$}
     \psfrag{ang2}[][][0.8]{$\ang_{p'}$}
     \psfrag{t}[][][0.8]{$\ang^*$}
     \psfrag{Dang1}[][][0.8]{$\Delta\ang_p$}
     \psfrag{Dang2}[][][0.8]{$\ \ \Delta\ang_{p'}$}
  \begin{center}
    {\includegraphics[width=.2\textwidth]{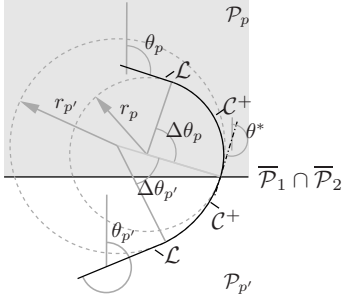}} 
\caption{Refraction law for $\L_p\C_{p-p'}\L_{p'}$-type paths at the boundary
for a vehicle traveling from $\P_p$ to $\P_{p'}$.
 The $\L$ path pieces define the angles $\ang_p, \ang_{p'}$ and their variations $\Delta\ang_p, \Delta\ang_{p'}$, which satisfy equations \eqref{eqn:RefractionHSv} and \eqref{eqn:RefractionHSv2},
which is a generalization of Snell's law of refraction.}
\label{fig:Refraction}
  \end{center}
\end{figure}

\ShowProofs{
\begin{proof}
We first carry out the coordinate transformation \eqref{eqn:coordTrans} at
the boundary of the $(j-1)$-th region and the $j$-th region at which the pair of
pieces $\X_{j-1},\X_j$ connect, in order that $\X_{j-1},\X_j$ connect at a point 
in the set $\defset{(\x,\y)\in\reals^2}{\y=0}$.
From items b) and c) of Lemma~\ref{propo:PropertiesTwoRegions}, we obtain $\adjx=\adjx^+=\overline{\adjx}$, $\adjy=\overline{\adjy}_p$, $\adjy^+=\overline{\adjy}_{p'}$, $\adjang=\adjang^+$, $u=\sign(\adjang)\overline{u}_p=\sign(\adjang)\frac{v_p}{r_p}$ and $u^+=\sign(\adjang^+)\overline{u}_{p'}=\sign(\adjang)\frac{v_{p'}}{r_{p'}}$. Moreover, at the boundary, we have $\ang=\ang^+=\ang^*$. Thus, \eqref{eqn:HamiltonianJump2} and \eqref{eqn:HamiltonianJump3} become 
\begin{align}\label{eqn:HamiltonianContinuitynoInputAngStar}
\nonumber &\overline{\adjx} v_{p} \sin \ang^*+\overline{\adjy}_p v_{p} \cos \ang^* + |\adjang| \frac{v_p}{r_p} \\ &= 
\overline{\adjx} v_{p'} \sin \ang^*+\overline{\adjy}_{p'} v_{p'} \cos \ang^* + |\adjang| \frac{v_{p'}}{r_{p'}}=c+\adj_{\circ} \ ,\\ \label{eqn:HamiltonianContinuitynoInputAngStar2}
\nonumber &\overline{\adjx} v_{p} \sin \ang_p+\overline{\adjy}_p v_{p} \cos \ang_p  \\&= \overline{\adjx} v_{p'} \sin \ang_{p'}+\overline{\adjy}_{p'} v_{p'} \cos \ang_{p'}=c+\adj_{\circ} \ ,
\end{align}
respectively.
Furthermore, for the path pieces of type $\L$,
$\overline{\adjy}_p\tan \ang_p = \overline{\adjx}$ 
and
$\overline{\adjy}_{p'} \tan \ang_{p'} = \overline{\adjx}$ by 
item d) of Lemma~\ref{propo:PropertiesTwoRegions}. Substituting these into
\eqref{eqn:HamiltonianContinuitynoInputAngStar2},
we obtain \eqref{eqn:RefractionHSv}
since
\begin{align*}
&\overline{\adjx} v_{p} \sin \ang_p+\overline{\adjx} v_{p} \cot \ang_p \cos \ang_p  \\&= \overline{\adjx} v_{p'} \sin \ang_{p'}+\overline{\adjx} v_{p'}  \cot \ang_{p'} \cos \ang_{p'} \; \Rightarrow \; \frac{v_p}{v_{p'}}=\frac{\sin \ang_p}{\sin \ang_{p'}}\ .
\end{align*}
To finish the proof, we subtract \eqref{eqn:HamiltonianContinuitynoInputAngStar2} from \eqref{eqn:HamiltonianContinuitynoInputAngStar} and substitute $\overline{\adjy}_p\tan \ang_p = \overline{\adjx}$ 
and
$\overline{\adjy}_{p'} \tan \ang_{p'} = \overline{\adjx}$ in it. We get
\begin{align*}
&|\adjang| \frac{v_p}{r_p} \\
&=-\overline{\adjx} v_{p} (\sin \ang^*-\sin \ang_p)-\overline{\adjx} v_{p} \cot \ang_p ( \cos \ang^*- \cos \ang_p)\\
&=-\frac{\overline{\adjx} v_{p}}{\sin \ang_p}(\sin \ang_p (\sin \ang^*-\sin \ang_p)\\&\qquad + \cos \ang_p ( \cos \ang^*- \cos \ang_p))
\end{align*}
and
\begin{align*}
& |\adjang| \frac{v_{p'}}{r_{p'}}
\\&= -\overline{\adjx} v_{p'} (\sin \ang^*-\sin \ang_{p'})-\overline{\adjx} v_{p'} \cot \ang_{p'} ( \cos \ang^*- \cos \ang_{p'})\\
&=-\frac{\overline{\adjx} v_{p'}}{\sin \ang_{p'}} (\sin \ang_{p'} (\sin \ang^*-\sin \ang_{p'})\\ & \qquad+ \cos \ang_{p'} ( \cos \ang^*- \cos \ang_{p'}))\ .
\end{align*}
Dividing the latter equation above with the former, 
using 
\eqref{eqn:RefractionHSv},
 and simplifying, we arrive at \vspace{-0.5cm} \small
\begin{align*}
&\frac{r_{p}}{r_{p'}} \\& =\frac{\sin \ang_p}{\sin \ang_{p'}}\frac{\sin \ang_{p'} (\sin \ang^*-\sin \ang_{p'})+ \cos \ang_{p'} ( \cos \ang^*- \cos \ang_{p'})}{ \sin \ang_{p} (\sin \ang^*-\sin \ang_{p})+ \cos \ang_{p} ( \cos \ang^*- \cos \ang_{p})}\\&=\frac{\sin \ang_{p}(1-\cos \Delta \ang_{p'})}{\sin \ang_{p'}(1-\cos \Delta \ang_{p})}=\frac{v_p}{v_{p'}}\left(\frac{1-\cos \Delta \ang_{p'}}{1-\cos \Delta \ang_{p}}\right)\ ,
\end{align*}
\normalsize

which is  \eqref{eqn:RefractionHSv2}.
When the type of the path intersecting the boundary between $\P_p$ and $\P_{p'}$ is $\L$, 
by item d) of Theorem~\ref{thm:optimalityConditionsPv}, $\ang_p = \ang_{p'} = 0 \pmod{\pi}$.
\end{proof}

\begin{corollary}
Under the conditions of Theorem~\ref{thm:refractionPv},
if $\overline{u}_p=\overline{u}_{p'}$} (Problem~\ref{problem:NRegions}$\star$), then \eqref{eqn:RefractionHSv2} reduces to $\Delta \ang_p=\Delta \ang_{p'}$.
\end{corollary}
\ShowProofs{
\begin{proof}
The condition in Problem 1$\star$ is given by $\frac{v_p}{r_p} = \frac{v_{p'}}{r_{p'}}$. Using this in \eqref{eqn:RefractionHSv2}, we get $\Delta \theta_p = \pm \Delta \theta_{p'}$. To rule out the case  $\Delta \theta_p = - \Delta \theta_{p'}$, suppose, by contradiction, that this relation holds. Using the definition of $\Delta \theta_p$ and $\Delta \theta_{p'}$, we get $\theta_p = \theta_{p'}$. Then, applying \eqref{eqn:RefractionHSv}, the velocities $v_p$ and $v_{p'}$ in $\P_p$ and $\P_{p'}$, respectively, are equal. Replacing this in the condition in Problem 1$\star$ we get that $r_p = r_{p'}$. This leads to $\P_p$ and $\P_{p'}$ having the same properties, which is a contradiction.}
\end{proof}

Equations~\eqref{eqn:RefractionHSv} and ~\eqref{eqn:RefractionHSv2}
can be interpreted as a refraction law at the boundary of the two regions for the 
angles (and their variations) $\ang_p, \ang_{p'}$ (and $\Delta \ang_p,\Delta \ang_{p'}$).
In the case that the particle is allowed to instantaneously change its direction of travel, this refraction law
simplifies to the so-called {\em Snell's law of refraction}.
Snell's law of refraction in optics states a relationship between the 
angles of rays of light when passing through the boundary of two isotropic media with different refraction coefficients.
More precisely, given two media defining two regions in the state space with different refraction indices $v_1$ and $v_2$,
Snell's law of refraction states that
\begin{eqnarray}\label{eqn:SnellsLaw}
\frac{v_1}{v_2}=\frac{\sin\ang_1}{\sin\ang_2}\ ,
\end{eqnarray}
where $\ang_1$ is the angle of incidence and $\ang_2$ is the angle of refraction.
This law can be derived by solving a minimum-time problem between two points, one in 
each medium. Moreover, the dynamics of the rays of light can be associated to the
differential equations $\dot{x} = v_i$, where $v_i$ is the velocity in the $i$-th medium, $i = 1,2$. 
In fact,  Snell's law can be seen as a special case for refraction with the dynamics given by  \eqref{eqn:Dubins} with a radius of curvature of zero, i.e. with instantaneous jumps in $\ang$. In this case, \eqref{eqn:RefractionHSv2} 
holds trivially for arbitrary values of $\Delta \ang_1$ and $\Delta \ang_2$,
 while  the ratio of velocities 
 in \eqref{eqn:RefractionHSv} is essentially
\eqref{eqn:SnellsLaw}.
 Thus, Theorem~\ref{thm:refractionPv} generalizes Snell's law to the case 
when the ``dynamics" of the rays of light are given by \eqref{eqn:Dubins}. In the context of autonomous vehicles, the conditions given by \eqref{eqn:RefractionHSv} and \eqref{eqn:RefractionHSv2} are a generalization of the refraction law
for optimal steering of a point-mass vehicle, as in \cite{AlexanderRowe90,RoweAlexander00},
to the Dubins vehicle case.

\subsection{Optimal families of paths}

\begin{figure}[Htp!] 
  \begin{center}
     \psfrag{Cp}[][][0.9]{$\C^+$}
     \psfrag{Cm}[][][0.9]{$\C^-$}
     \psfrag{L}[][][0.9]{$\L$}
     \psfrag{M}[][][0.9]{}
     \psfrag{1}[][][0.8]{${\cal P}_p$}
     \psfrag{2}[][][0.8]{${\cal P}_{p'}$}
     \psfrag{R1}[][][0.8]{$r_p$}
     \psfrag{R2}[][][0.8]{$r_{p'}$}
      \subfigure[$\L$-type of path perpendicular to boundary. \label{fig:OptimalPathsAtBoundaryL}]
    {\includegraphics[width=.14\textwidth]{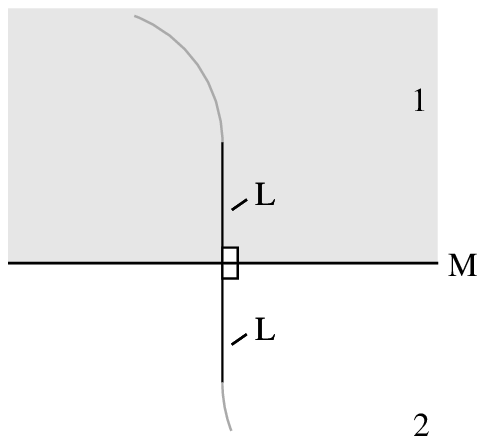}} \
    \subfigure[$\C^+$-type of path
 \label{fig:OptimalPathsAtBoundaryCplus1}]
    {\includegraphics[width=.14\textwidth]{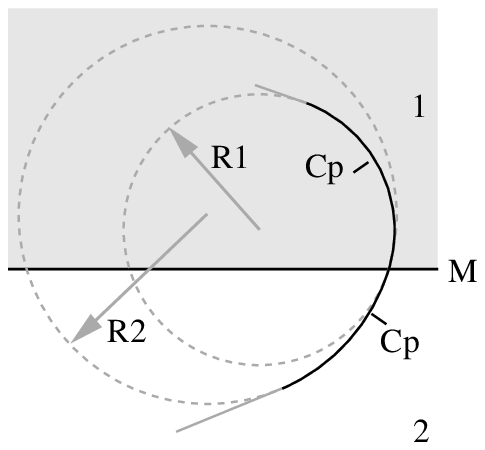}} \ 
     \subfigure[$\C^-$-type
 \label{fig:OptimalPathsAtBoundaryCplus2}]
    {\includegraphics[width=.14\textwidth]{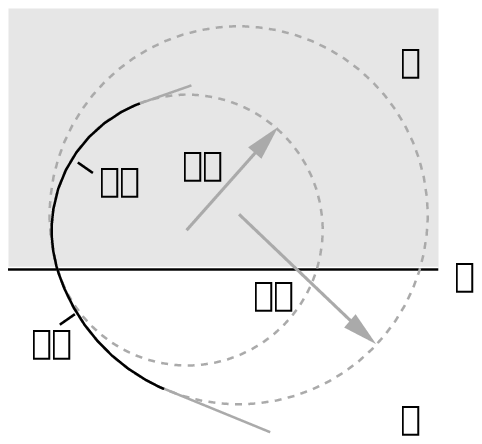}}  
    \subfigure [$\C^+/\L$-type
 \label{fig:OptimalPathsAtBoundaryCplus3}]
    {\includegraphics[width=.14\textwidth]{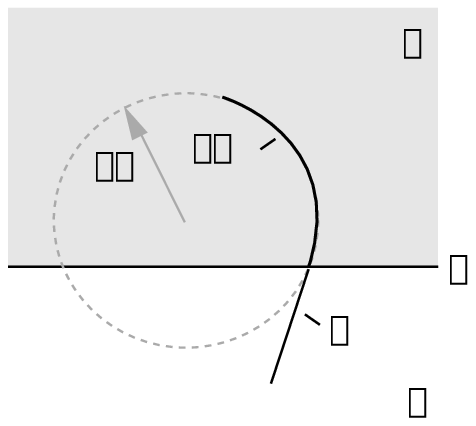}} \
     \subfigure [$\C^-/\L$-type
     \label{fig:OptimalPathsAtBoundaryCplus4}]
    {\includegraphics[width=.14\textwidth]{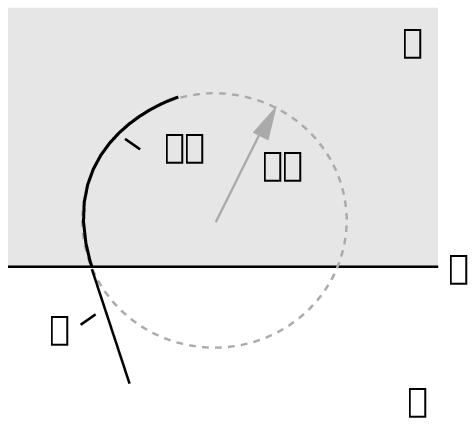}} \
    \subfigure[$\L/\C^+$-type
    \label{fig:OptimalPathsAtBoundaryL1}]
    {\includegraphics[width=.14\textwidth]{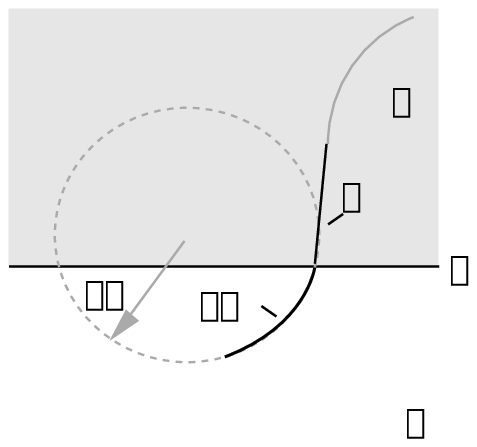}} \\
     \subfigure[$\L/\C^-$-type
     \label{fig:OptimalPathsAtBoundaryL2}]
    {\includegraphics[width=.14\textwidth]{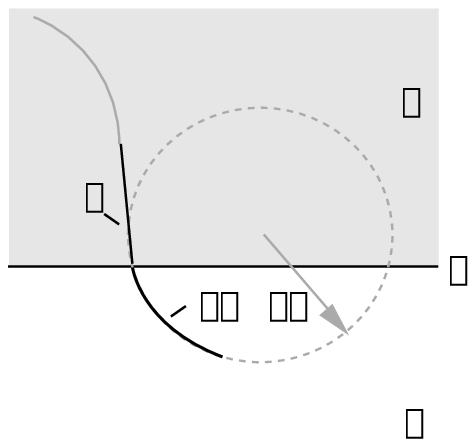}} \
    \subfigure[$\C^+/\C^-$-type
    \label{fig:OptimalPathsAtBoundaryL3}]
    {\includegraphics[width=.14\textwidth]{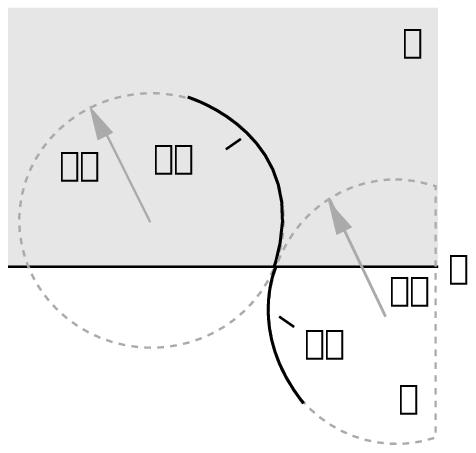}} \
        \subfigure[$\C^-/\C^+$-type
        \label{fig:OptimalPathsAtBoundaryL4}]
    {\includegraphics[width=.14\textwidth]{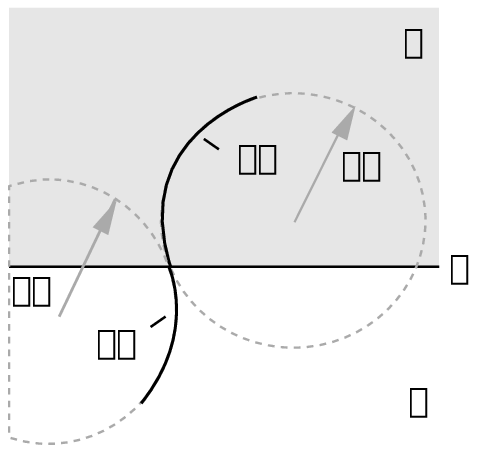}}
    
\caption{Optimal paths satisfying the necessary conditions in Theorem~\ref{thm:optimalityConditionsPv} at the boundary between two regions with minimum turning radius $r_{p}$ in region ${\cal P}_{p}$ and with minimum turning radius $r_{p'}$ in region ${\cal P}_{p'}$, $p,p' \in \Pset$. }
\label{fig:OptimalPathsAtBoundary}
  \end{center}
\end{figure}

\begin{figure}[tp!] 
  \begin{center}
     \psfrag{Cm}[][][0.9]{$\C^-$}
     \psfrag{Cp}[][][0.9]{$\C^+$}
     \psfrag{L}[][][0.9]{$\L$}
     \psfrag{M}[][][0.9]{}
     \psfrag{1}[][][0.8]{${\cal P}_p$}
     \psfrag{2}[][][0.8]{${\cal P}_{p'}$}
     \psfrag{R1}[][][0.8]{$r_p$}
     \psfrag{R2}[][][0.8]{$r_{p'}$}
\subfigure[$\L$-type of path nonorthogonal to boundary.]
    {\includegraphics[width=.14\textwidth]{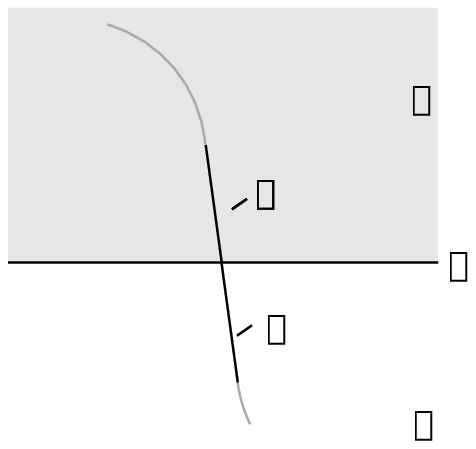}} \
    \subfigure[$\L_p\C^+_p\L_{p'}$-type.]
    {\includegraphics[width=.14\textwidth]{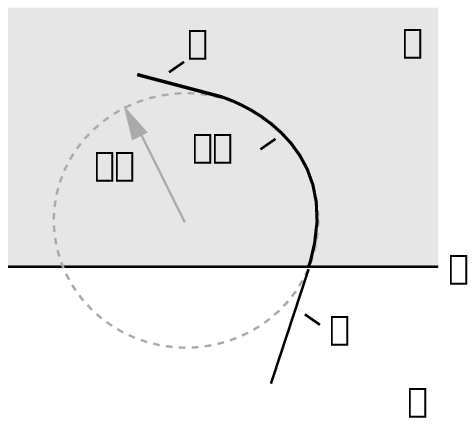}} \
    \subfigure[$\L_p\C^-_p\L_{p'}$-type.]
    {\includegraphics[width=.14\textwidth]{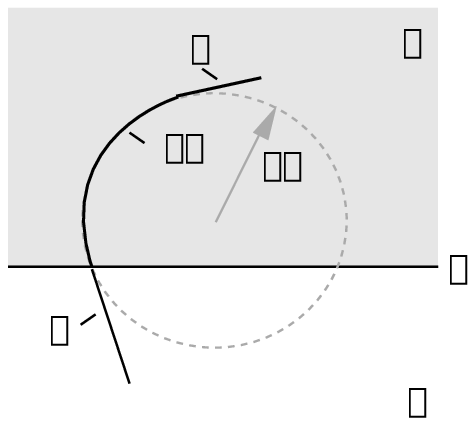}} \\
    \subfigure[$\L_p\C^+_{p'}\L_{p'}$-type.]
    {\includegraphics[width=.14\textwidth]{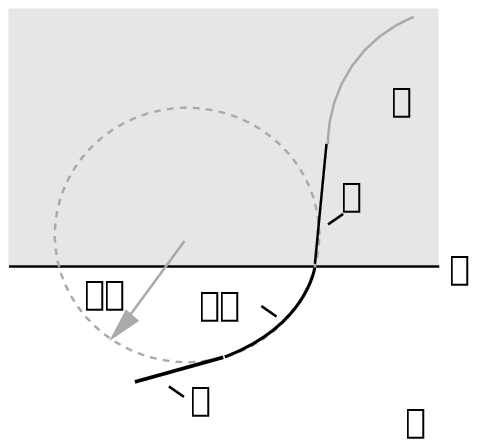}} \
    \subfigure[$\L_p\C^-_{p'}\L_{p'}$-type.]
    {\includegraphics[width=.14\textwidth]{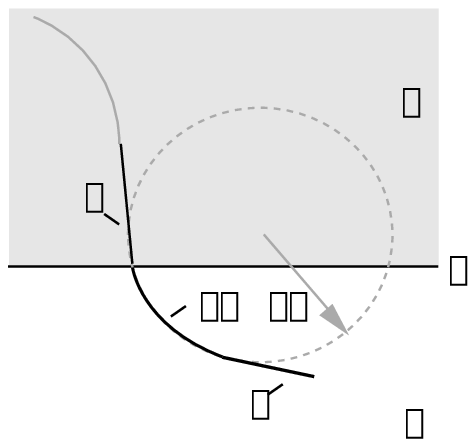}} \\
    \subfigure[$\L_p\C^+_{p}\C^-_{p'}\L_{p'}$-type.]
    {\includegraphics[width=.14\textwidth]{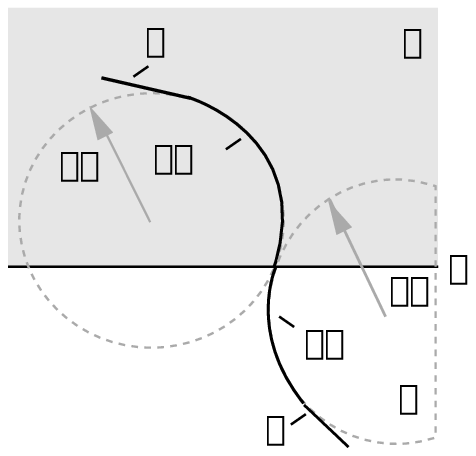}} \
    \subfigure[$\L_p\C^-_{p}\C^+_{p'}\L_{p'}$-type.]
    {\includegraphics[width=.14\textwidth]{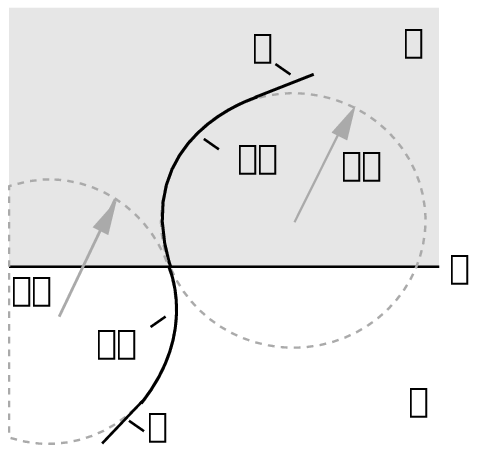}} \
\caption{Nonoptimal paths at the boundary determined by Corollary~\ref{coro:NonOptimalPaths}.}
\label{fig:NonOptimalPathsAtBoundary}
  \end{center}
\end{figure}

Using 
Theorem~\ref{thm:optimalityConditionsPv}, it is possible to determine optimal families of
paths for a class of paths to Problem~\ref{problem:NRegions}. The following statements follow directly from Dubins' 
result and Theorem~\ref{thm:optimalityConditionsPv}.
Below, 
the subscript $p$, $p'$ or $p-p'$ on each path piece denotes that the path piece is in region $p$, $p'$  or spans across both regions $p$ and $p'$, respectively. The numeric subscript on each path piece indicates the number of path piece. 

\begin{corollary}{(optimal paths w/one jump)}
\label{coro:OptimalPaths}
Let the curve $\X$ describe a minimum-length path that solves Problem~\ref{problem:NRegions} and let $\x, \y$ and  $\ang$ 
be its associated functions with input $u$.
Define the function $q$ 
following the 
construction below  
 \eqref{eqn:SwitchingSet}.
 Suppose Assumption~\ref{assump:OptimalTrajectories} holds.
If $\X$ 
has no more than one boundary cross between two adjacent regions $p$ and $p'$, then it is a smooth concatenation 
of $\C,\L$ paths pieces with at most three pieces in each region and is given by one of the following eight path types of paths:
\begin{align}\label{eqn:PVoptimalPaths}
\nonumber 
&\C_{1,p}\C_{2,p}\C_{3,p}\C_{4,p'}\L_{5,p'}\C_{6,p'},\ 
\C_{1,p}\C_{2,p}\C_{3,p}\C_{4,p'}\C_{5,p'}\C_{6,p'}, \\
\nonumber& \C_{1,p}\L_{2,p}\C_{3,p}\C_{4,p'}\C_{5,p'}\C_{6,p'} ,\
\C_{1,p}\L_{2,p}\C_{3,p-p'}\L_{4,p'}\C_{5,p'}, \\
\nonumber &\C_{1,p}\C_{2,p}\C_{3,p-p'}\L_{4,p'}\C_{5,p'},\ 
\C_{1,p}\C_{2,p}\C_{3,p-p'}\C_{4,p'}\C_{5,p'}, \\ 
& \C_{1,p}\L_{2,p}\C_{3,p-p'}\C_{4,p'}\C_{5,p'} ,\ 
\C_{1,p}\L_{2,p-p'}\C_{3,p'}\ ,
\end{align}
where $\L_{2,p-p'}$ is perpendicular to the boundary, in addition to any such path obtained when some of the path pieces (but not all) 
have zero length. 
\end{corollary}

Figure~\ref{fig:OptimalPathsAtBoundary} 
depicts families of optimal paths across the boundary of two regions with minimum turning radius $r_{p}$ in region ${\cal P}_{p}$ and with minimum turning radius $r_{p'}$ in region ${\cal P}_{p'}$, $p,p' \in \Pset$.
Item d) in Theorem~\ref{thm:optimalityConditionsPv} implies that $\L$-type paths at the boundary of regions with different velocity are optimal
only if they are orthogonal to 
the boundary, depicted by
Figure~\ref{fig:OptimalPathsAtBoundaryL}.

Another useful corollary is the following. 

\begin{corollary}{(nonoptimal paths)}
\label{coro:NonOptimalPaths}
Let the curve $\X$ describe a minimum-length path that solves Problem~\ref{problem:NRegions} and let $\x, \y$ and  $\ang$ 
be its associated functions with input $u$.
Define the function $q$ 
following the 
construction below  
 \eqref{eqn:SwitchingSet}.
 Suppose Assumption~\ref{assump:OptimalTrajectories} holds.
For every one boundary cross between two adjacent regions $p$ and $p'$, the minimum-time path cannot be a smooth concatenation 
of $\C,\L$ paths pieces with more than three pieces in each region and cannot belong to the following four types of paths:
\begin{align}\label{eqn:PVNonoptimalPaths}
\nonumber& \C_{1,p}\L_{2,p-p'}\C_{3,p} , \
\C_{1,p}\L_{2,p}\C_{3,p}\L_{4,p'}\C_{5,p'},\\  
& \C_{1,p}\L_{2,p}\C_{3,p'}\L_{4,p'}\C_{5,p'}, \
\C_{1,p}\L_{2,p}\C_{3,p}\C_{4,p'}\L_{5,p'}\C_{6,p'} , 
\end{align}
where $\L_{2,p-p'}$ is non-orthogonal to the boundary, in addition to any such path obtained when the first and/or last path pieces
have zero length. 
\end{corollary}

\ShowProofs{
\begin{proof}
From the results in \cite{Dubins57},
any path with more than three path pieces in each region is nonoptimal.
 In addition, by item d) in Theorem~\ref{thm:optimalityConditionsPv}, 
an $\L_{p-p'}$-type of the path intersecting the boundary between $P_p$ and $P_{p'}$	must be orthogonal to the boundary. 
This rules out the first path in \eqref{eqn:PVNonoptimalPaths}.
The other three  paths in \eqref{eqn:PVNonoptimalPaths} have segments of type $\L_p\C_p\L_{p'}$, $\L_p\C_{p'}\L_{p'}$ or $\L_p\C_p\C_{p'}\L_{p'}$, which 
have a switch of the input $u$ at the boundary and correspond to a transition from or to a path 
of type $\C$.  
This implies that $\adjang=0$ at the boundary since, by item b) in Lemma~\ref{propo:PropertiesTwoRegions}, it is a continuous function. 
To reach a contradiction, suppose these path segments are optimal. Then, $\adjang=0$ at the boundary implies that \eqref{eqn:HamiltonianJump2} becomes 
\begin{eqnarray}\label{eqn:HamiltonianContinuitynoInputAngStar3}
\nonumber \overline{\adjx} v_{p} \sin \ang^*+\overline{\adjy}_p v_{p} \cos \ang^* &=& 
\overline{\adjx} v_{p'} \sin \ang^*+\overline{\adjy}_{p'} v_{p'} \cos \ang^* \\ &=&c+\adj_{\circ} \ .
\end{eqnarray}
Then, by subtracting \eqref{eqn:HamiltonianContinuitynoInputAngStar2} from \eqref{eqn:HamiltonianContinuitynoInputAngStar3} and substituting $\overline{\adjy}_p\tan \ang_p = \overline{\adjx}$ 
and
$\overline{\adjy}_{p'} \tan \ang_{p'} = \overline{\adjx}$ (by item d) in Lemma~\ref{propo:PropertiesTwoRegions}), we get
\begin{align*}
&\frac{\overline{\adjx} v_{p}}{\sin \ang_p}(\sin \ang_p (\sin \ang^*-\sin \ang_p)+ \cos \ang_p ( \cos \ang^*- \cos \ang_p))\\
&= \frac{\overline{\adjx} v_{p}}{\sin \ang_p}(1-\cos \Delta \ang_p)  = 0\\
&\frac{\overline{\adjx} v_{p'}}{\sin \ang_{p'}} (\sin \ang_{p'} (\sin \ang^*-\sin \ang_{p'})+ \cos \ang_{p'} ( \cos \ang^*- \cos \ang_{p'}))\\&= \frac{\overline{\adjx} v_{p'}}{\sin \ang_{p'}}(1-\cos \Delta \ang_{p'})  = 0\ ,
\end{align*}
which implies that $\Delta \ang_p=\Delta \ang_{p'}=0$, i.e., no path piece of type $\C$ is allowed before or after the boundary for optimal paths, which leads to a contradiction.
\end{proof}}

Figure~\ref{fig:NonOptimalPathsAtBoundary}
depicts 
the path types
 that Corollary~\ref{coro:NonOptimalPaths}
determines to be nonoptimal.

\section{Numerical validation}
\label{sec:Example}

\subsection{Two Heterogeneous Regions}\label{sec:Example1}
For the purposes of illustrating our main results, we first consider the case
of two regions, $\P_1$ and $\P_2$, where 
$\P_1 = \defset{(\x,\y)\in\reals^2}{y> 0}$ and
$\P_2 = \defset{(\x,\y)\in\reals^2}{y\leq 0}$, and
choose initial and final configurations, given by
$(x,y)^i \in \P_1^{\circ}$ with initial velocity vector $\nu^i$ and
$(x,y)^f \in \P_2^{\circ}$ with final velocity vector $\nu^f$, respectively,
so that optimal paths between the regions cross only once, that is, 
for every optimal path there exists a unique crossing point
$(x^*,y^*)\in \overline{\P}_1\cap \overline{\P}_2$ with a unique velocity vector direction. 
We use the off-the-shelf software package GPOPS \cite{Rao.10} to verify the necessary conditions for optimality as well as the refraction law at boundary derived in Section~\ref{sec:Results}. 
The simulations were performed on a 2.2 GHz Intel Core i7 CPU with 15 mesh refinement iterations, a mesh tolerance of $10^{-5}$, and between 25 to 50 nodes per interval. The average CPU time for each simulation was 83.3 s.

\begin{figure}[htp!] 
\begin{center}
\subfigure[Optimal paths with $r_1=v_1$, $r_2=v_2$ (Problem~\ref{problem:NRegions}$\star$). \tiny{Minimum times (from left to right): \{25.32s, 15.79s, 10.68s,  8.02s, 7.29s\}; Values of adjoints: $\overline{\adjx}=$\{-0.74, -0.59, -0.34, -0.09, 0.53\}, $\overline{\adjy}_1=$\{3.93, 1.91,0.94,0.49,0.03\}, $\overline{\adjy}_2$ =\{0.68, 0.81, 0.84, 1.00, 0.85\}, $\overline{\adjang}_{min}=$\{-1.56,-1.47,-1.42,-1.57,-3.09\}, $\overline{\adjang}_{max}$=\{0.0514, 0.0135, 4.16$\times10^{-5}$, 4.94$\times10^{-5}$, 4.42$\times10^{-5}$\} } \label{fig:classic}]
   {\includegraphics[width=.22\textwidth,trim=15mm 2mm 10mm 2mm]{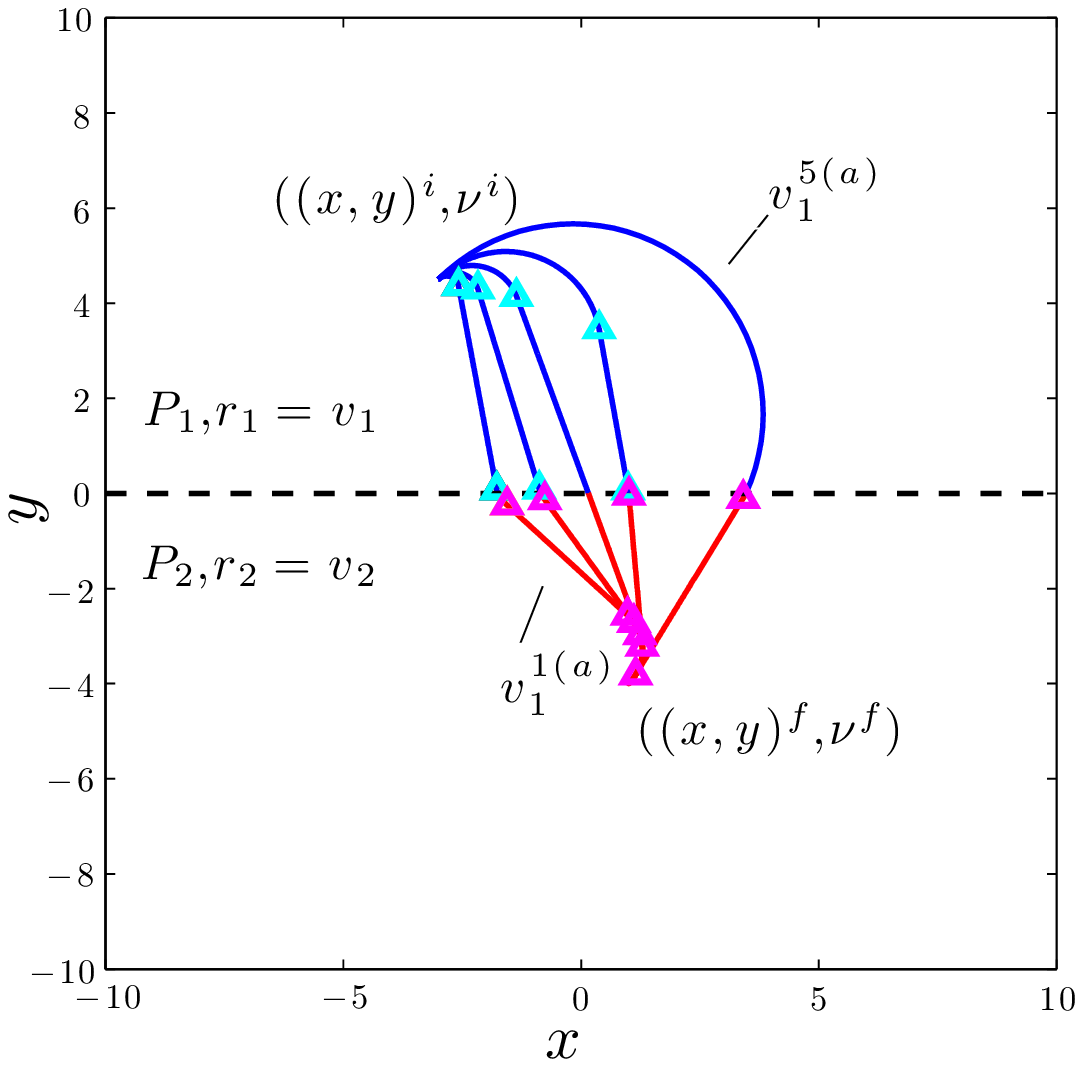}}
\quad  
\subfigure[Optimal paths with $r_1=0.5 v_1$, $r_2=v_2$. \tiny{Minimum times (from left to right): \{24.64s, 15.21s, 10.17s,  7.44s, 6.05s\} ; Values of adjoints: $\overline{\adjx}=$\{-0.72, -0.53, -0.25, 0.11, 0.90\}, $\overline{\adjy}_1=$\{3.93,\allowbreak 1.90, 0.91, 0.45, 0.24\}, $\overline{\adjy}_2=$\{0.65, 0.77, 0.91, 0.98, 1.00\}, $\overline{\adjang}_{min}=$\{-1.07, -0.90, -0.67, -0.67, -0.77\}, $\overline{\adjang}_{max}=$\{0.0378, 0.0114,4.86$\times10^{-5}$, 4.70$\times10^{-5}$, 3.98$\times10^{-5}$\}}   \label{fig:sim3b}]
   {\includegraphics[width=.22\textwidth,trim=15mm 2mm 10mm 2mm]{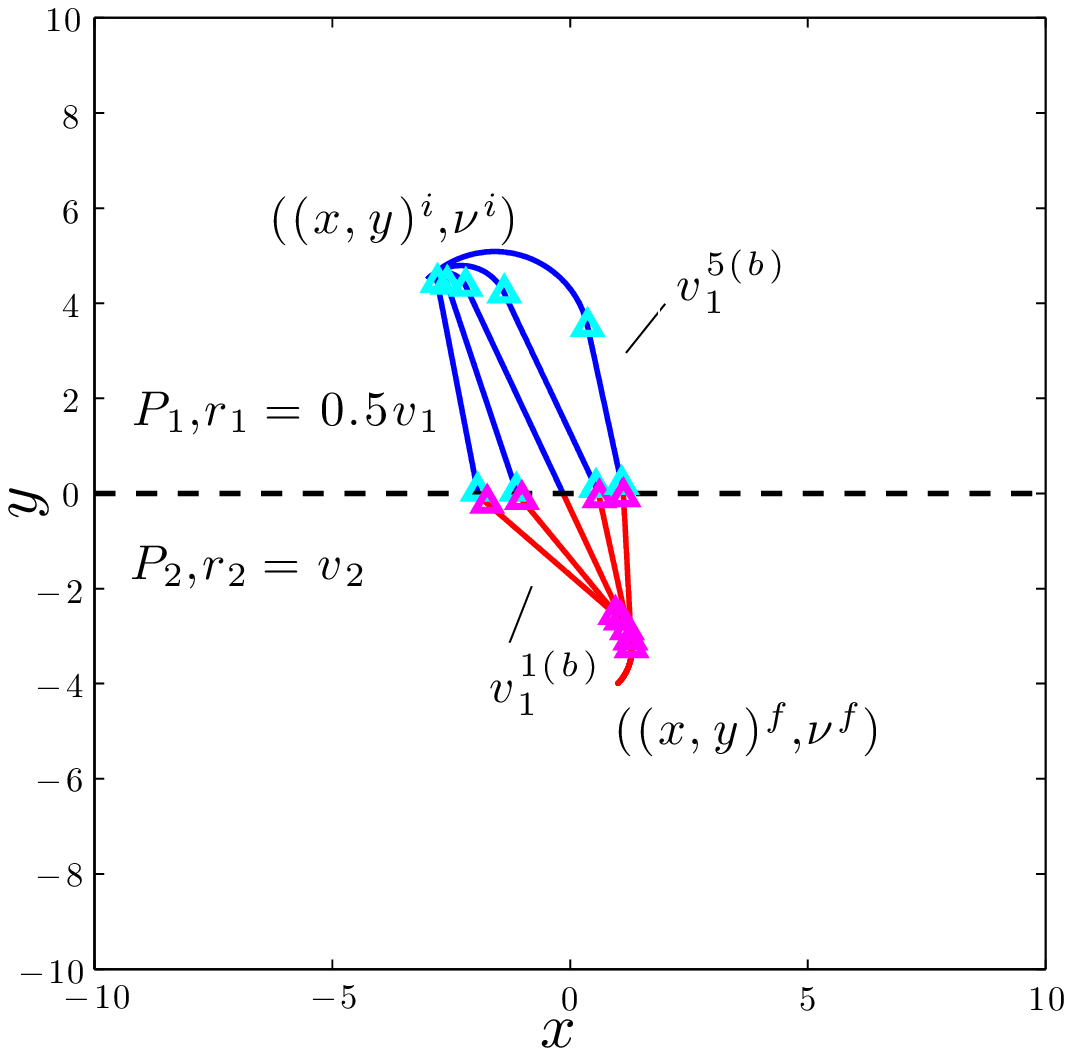}}
   \\
\subfigure[Optimal paths with $r_1=1.5 v_1$, $r_2=v_2$. \tiny{Minimum times (from left to right): \{25.99s , 16.39s, 11.27s,  8.85s, 13.53s\}; Values of adjoints: $\overline{\adjx}=$\{-0.72, -0.53, -0.25, 0.11, 0.90\}, $\overline{\adjy}_1=$\{3.94,\allowbreak 1.93, 0.97, 0.49, -0.50\}, $\overline{\adjy}_2=$\{0.70, 0.85, 0.97, 0.99, 0.44\}, $\overline{\adjang}_{min}=$\{-2.35, -2.24, -2.27, -2.77, -6.07\}, $\overline{\adjang}_{max}=$\{0.0574, 0.0130, 5.05$\times10^{-5}$, 0.0019, 1.5604\}}  \label{fig:sim3c}]
   {\includegraphics[width=.215\textwidth,trim=15mm 2mm 10mm 2mm]{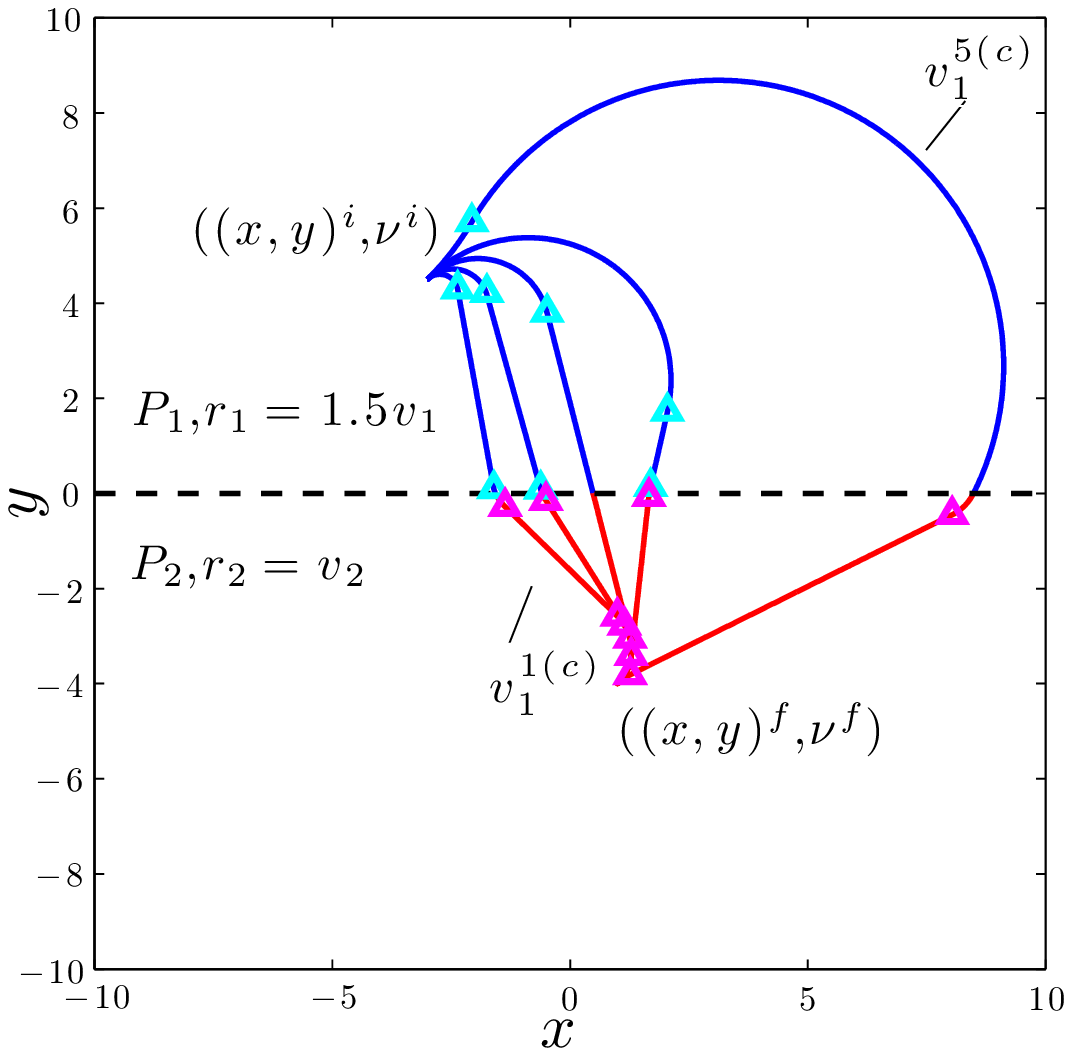}}
 \quad 
\subfigure[Optimal paths with $r_1=0.5 v_1$, $r_2=1.5 v_1$. \tiny{Minimum times (from left to right): \{24.28s, 15.10s, 10.32s,  7.79s, 6.62s\}; Values of adjoints: $\overline{\adjx}=$\{-0.66, -0.60, -0.46, -0.28, -0.15\}, $\overline{\adjy}_1=$\{3.94, 1.91, 0.89, 0.41, 0.20\}, $\overline{\adjy}_2=$\{0.75, 0.80, 0.89, 0.96, 0.99\}, $\overline{\adjang}_{min}=$\{-0.79, -0.73, -1.05, -1.56, -2.46\}, $\overline{\adjang}_{max}=$\{0.0166, 0.0087, 4.14$\times10^{-5}$, 2.08$\times10^{-5}$, 1.01$\times10^{-5}$\} } \label{fig:sim3d}]
   {\includegraphics[width=.22\textwidth,trim=15mm 2mm 10mm 2mm]{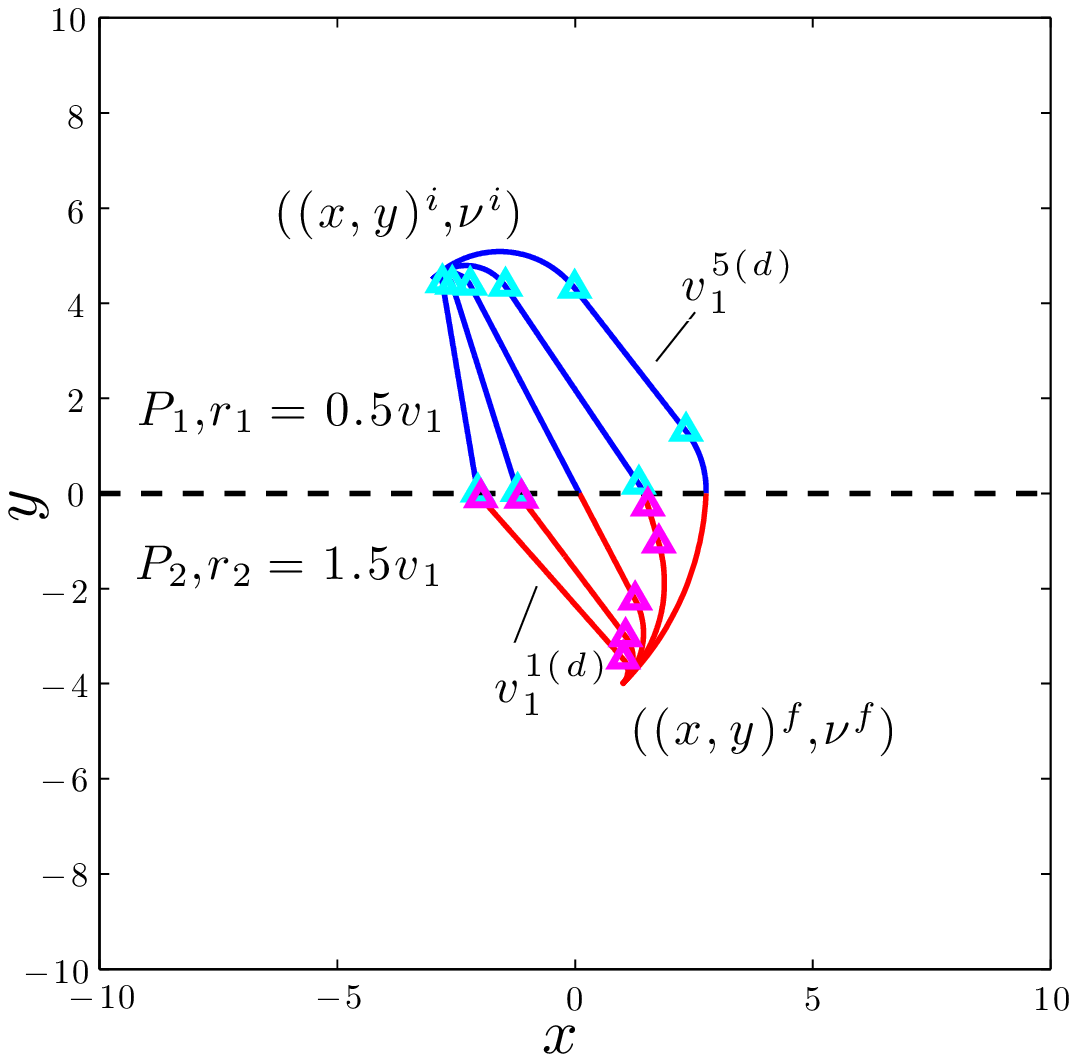}}
\caption{Optimal paths from 
a given initial point 
$((x,y)^i,\nu^i) = ((-3,4.5),(v_1, \angle \frac{\pi}{4}))$
to a final point 
$((x,y)^f,\nu^f) = ((1,-4),(v_2, \angle \frac{5\pi}{4}))$
for different values of the velocity $v_1=\{\frac{1}{4},\frac{1}{2},1,2,4\}$ (from left to right) and minimum turning radii $r_1$ in region $\P_1$, and a constant velocity given by $v_2=1$ and different minimum turning radii $r_2$ in region $\P_2$.
The path for velocity $v^{3(a)}_1=1$ is the classical minimum-time Dubins path. \ricardo{The symbols $\triangle$ mark the junctions of the bang-singular-bang controls, i.e., the locations where the path type changes}.} 
\label{fig:sim3}
  \end{center}
\end{figure}
Figure~\ref{fig:sim3} shows optimal paths from a given initial point to a final point for different values of the velocities, $v_1$ and $v_2$ and minimum turning radii, $r_1$ and $r_2$, in regions $\P_1$ and $\P_2$, respectively. The path for velocity $v^{3(a)}_1$ in Figure~\ref{fig:classic} is the classical minimum-time Dubins path
since $v_1^{3(a)} = v_2 = 1$ and $r_1^{3(a)} = r_2 = 1$.

\begin{table}[t!]
\begin{center}
\caption{Comparison of specified ratios of velocities, $\frac{v_1}{v_2}$, and minimum turning radii, $\frac{r_1}{r_2}$, with computed values from \eqref{eqn:RefractionHSv} and \eqref{eqn:RefractionHSv2}. Entries with ``-'' correspond to the cases when the subpath across the boundary is not of type $\L\C\L$ and, hence, the refraction law does not apply. \label{table:comparison}}
\tiny
    \begin{tabular}{c|cc|cc|cc|cc}
       \toprule
   ~&   ~   & ~ & ~ & ~ & \multicolumn{2}{c|}{Specified}                & \multicolumn{2}{c}{From \eqref{eqn:RefractionHSv}, \eqref{eqn:RefractionHSv2}}                \\
  ~&      $v_1$     & $r_1$     & $v_2$     & $r_2$     & $\frac{v_1}{v_2}$         & $\frac{r_1}{r_2}$ & $\frac{v_1}{v_2}$                                                               & $\frac{r_1}{r_2}$ \\  \midrule

(a) &        0.25  & 0.25  & 1     & 1     & 0.25                      & 0.25              & 0.250                                                                           & 0.250             \\ 
~&        0.5   & 0.5   & 1     & 1     & 0.5                       & 0.5               & 0.499                                                                           & 0.499             \\ 
~&        1     & 1     & 1     & 1     & 1                         & 1                 & 1.000                                                                           & -                 \\ 
~&        2     & 2     & 1     & 1     & 2                         & 2                 & 2.008                                                                           & 2.008             \\ 
~&        4     & 4     & 1     & 1     & 4                         & 4                 & -                                                                               & -                 \\  \midrule
(b) &       0.25  & 0.125 & 1     & 1     & 0.25                      & 0.125             & 0.250                                                                           & 0.126             \\ 
~&        0.5   & 0.25  & 1     & 1     & 0.5                       & 0.25              & 0.500                                                                           & 0.255             \\ 
~&        1     & 0.5   & 1     & 1     & 1                         & 0.5               & 1.000                                                                           & -                 \\ 
 ~&       2     & 1     & 1     & 1     & 2                         & 1                 & 2.003                                                                           & 1.016             \\ 
 ~&       4     & 2     & 1     & 1     & 4                         & 2                 & 4.014                                                                           & 2.084             \\  \midrule
(c) &        0.25  & 0.375 & 1     & 1     & 0.25                      & 0.375             & 0.250                                                                           & 0.372             \\ 
~&        0.5   & 0.75  & 1     & 1     & 0.5                       & 0.75              & 0.499                                                                           & 0.742             \\ 
 ~&       1     & 1.5   & 1     & 1     & 1                         & 1.5               & 1.000                                                                           & -                 \\ 
 ~&       2     & 3     & 1     & 1     & 2                         & 3                 & 1.993                                                                           & 2.966             \\ 
~&        4     & 6     & 1     & 1     & 4                         & 6                 & -                                                                               & -                 \\   \midrule
(d) &        0.25  & 0.125 & 1     & 0.375 & 0.25                      & 0.333             & 0.250                                                                           & 0.336             \\ 
 ~&       0.5   & 0.25  & 1     & 0.75  & 0.5                       & 0.333             & 0.500                                                                           & 0.335             \\ 
 ~&       1     & 0.5   & 1     & 1.5   & 1                         & 0.333             & 1.000                                                                           & -                 \\ 
~&        2     & 1     & 1     & 3     & 2                         & 0.333             & 1.999                                                                           & 0.336             \\ 
 ~&       4     & 2     & 1     & 6     & 4                         & 0.333             & -                                                                               & -                 \\
     \bottomrule
    \end{tabular}
    \end{center}
\end{table}

Paths $v_1^{3(b)}$, $v_1^{3(c)}$ and $v_1^{3(d)}$ in Figures~\ref{fig:sim3b}, ~\ref{fig:sim3c} and ~\ref{fig:sim3d} have equal velocities of travel in both regions, 
and are very similar to the classical Dubins path, as the path piece across the boundary can be of type $\L$ and non-orthogonal to the boundary, i.e., with ``no refraction," resulting in paths of type $\C^+\L\C^+$. The only difference lies in the curvature constraint of the paths in the different regions.

Paths with smaller and larger velocities $v_1 \in \{\frac{1}{4},\frac{1}{2},2,4\}$ are also depicted showing that the particle travels a larger distance
in the region with the larger velocity of travel, as one would expect. Similarly, a larger minimum turning radius in each region leads to larger traveled distance. Turns of types $\C^-$ and $\C^+$ occur at points nearby the boundary between the regions for optimal paths $\{v_1^{1(a)}, v_1^{2(a)}, v_1^{1(b)}, v_1^{2(b)},v_1^{1(c)}, v_1^{2(c)}, v_1^{4(c)}, v_1^{1(d)},v_1^{2(d)}\}$ and $\{v_1^{4(a)}, v_1^{5(a)}, v_1^{4(b)}, v_1^{5(b)},v_1^{5(c)}, v_1^{4(d)},v_1^{5(d)}\}$, respectively. With the exception of $\{v_1^{5(a)}, v_1^{5(c)},v_1^{5(d)}\}$, these paths have a subpath of type $\L\C\L$ across the boundary and can be used to verify the derived refraction law given in \eqref{eqn:RefractionHSv} and \eqref{eqn:RefractionHSv2}. A comparison of the specified ratios of velocities, $\frac{v_1}{v_2}$, and minimum turning radii, $\frac{r_1}{r_2}$ with the computed values using the refraction law is shown in Table \ref{table:comparison}.  This table demonstrates an agreement of the simulation results with theory, where the slight differences can be attributed to the numerical imprecision of the optimal control solver.

\subsection{Three Heterogeneous Regions}\label{sec:Example2}

\begin{figure}[htp!] 
\begin{center}
\subfigure[Optimal paths with $\{v_1=0.25,r_1=0.5\}$ in $\P_1$;\newline $\{v_2=2, r_2=1\}$ in $\P_2$; and $\{v_3=0.25, r_3=0.6\}$ in $\P_3$.  \label{fig:lanechange}]{ \includegraphics[width=0.44\textwidth,trim=55mm 1mm 25mm 1mm,clip]{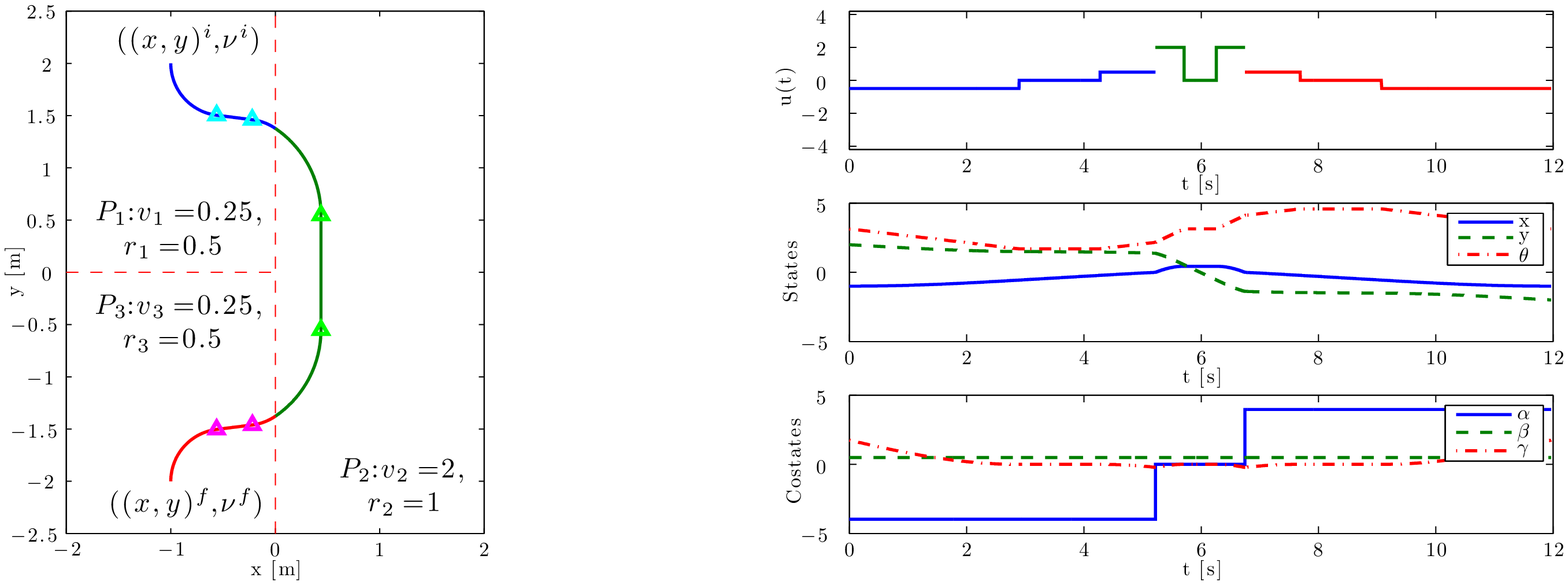}}
\subfigure[Optimal paths with $\{v_1=2,r_1=1\}$ in $\P_1$; \newline$\{v_2=0.5, r_2=0.8\}$ in $\P_2$; and $\{v_3=1, r_3=0.5\}$ in $\P_3$.  \label{fig:straight}]
   {\includegraphics[width=0.44\textwidth,trim=55mm 1mm 25mm 1mm,clip]{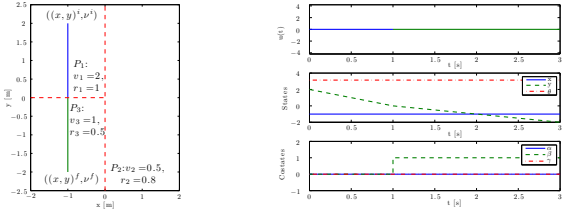}}
\subfigure[Optimal paths with $\{v_1=0.25,r_1=0.75\}$ in $\P_1$;\newline $\{v_2=2, r_2=0.5\}$ in $\P_2$; and $\{v_3=0.5, r_3=0.3\}$ in $\P_3$.  \label{fig:right}]
   {\includegraphics[width=0.44\textwidth,trim=55mm 1mm 25mm 1mm,clip]{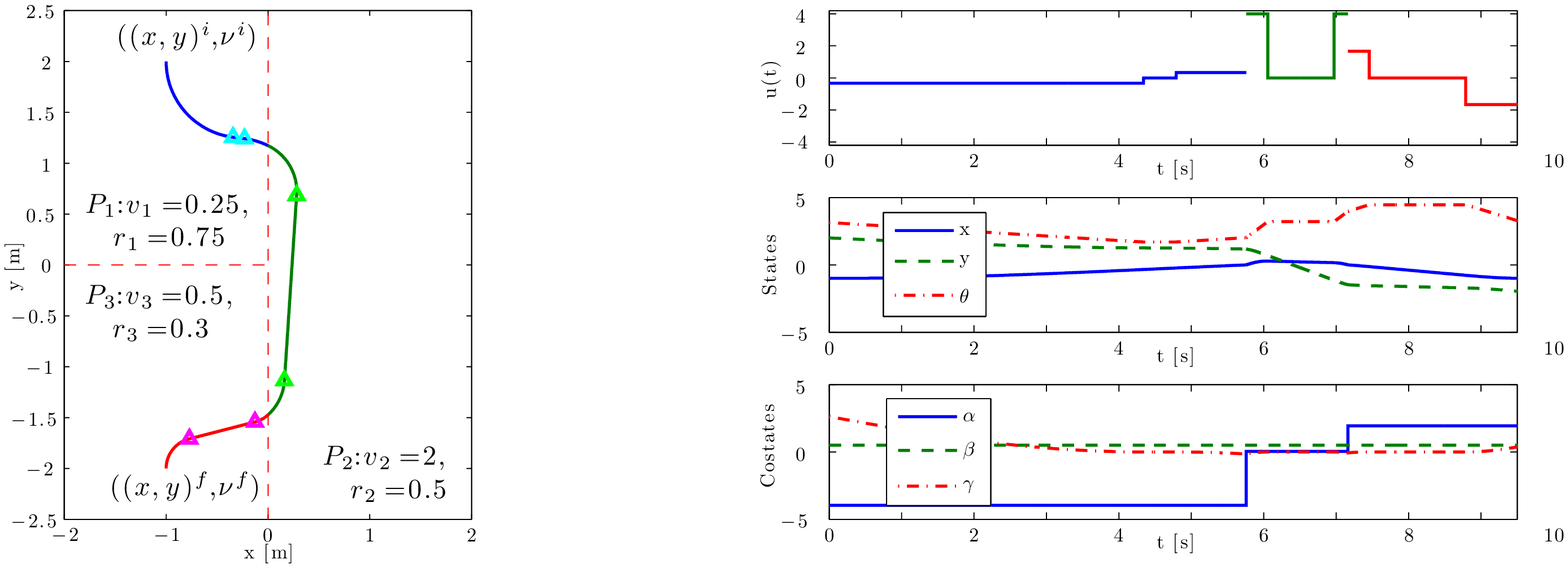}}
\caption{
Optimal paths from 
a given initial point 
$((x,y)^i,\nu^i) = ((-1,2),(v_1,\angle \frac{3\pi}{2}))$
to a final point 
$((x,y)^f,\nu^f) = ((-1,-2),(v_3,\angle \frac{3\pi}{2}))$
for different values of the velocity and minimum turning radius in each of the three regions, $\P_1$, $\P_2$ and $\P_3$, and history of control inputs (filtered), states and adjoint states along the optimal paths. \ricardo{The symbols $\triangle$ mark the junctions of the bang-singular-bang controls, i.e., the locations where the path type changes.}}
\label{fig:sim2}
  \end{center}
\end{figure}

The same optimal control solver that was used to compute optimal paths when there is only one crossing can also be  applied to the case of $N$ regions. For the general case of multiple crossings, computation is more involved, and the ability of the solver to find the global optimum strongly depends on having good initial guesses. We demonstrate this with an example with three regions, each with different velocities of travel and minimum turning radii, for a set of three different settings. 
The simulations were implemented on a 2.2 GHz Intel Core i7 CPU with 25 mesh refinement iterations, a mesh tolerance of $10^{-6}$, and  25 -- 35 nodes per interval. The average CPU time for each simulation was 60.4 s.

Figure~\ref{fig:sim2} shows optimal paths from a given initial point $((x,y)^i,\nu^i) = ((-1,2),(v_1,\angle \frac{3\pi}{2}))$ in region $\P_1$ to a final point  $((x,y)^f,\nu^f) = ((-1,-2),(v_3,\angle \frac{3\pi}{2}))$ in region $\P_3$. In this case, the particle has the option of moving directly from region $\P_1$ to $\P_3$ or passing through region $\P_2$ on the way to the final point in $\P_3$. In fact, the setting in Figure~\ref{fig:lanechange} is equivalent to an instance of a lane changing problem, where there are two lanes with different coefficient of frictions and a driver needs to decide whether to stay on the slow lane or to take a detour on the fast lane. Regions $\P_1$ and $\P_3$ are identical in size, geometry, and properties, and the union of the two regions can be seen as the slow lane, whereas the region on the right, $\P_2$ is the fast lane. For the setting in Figure~\ref{fig:lanechange}, the particle travels through region $\P_2$ because, although the traveled distance is increased, the travel time is decreased. The control input history also shows that the control $u$ is of the ``bang-singular-bang" family as shown in Theorem~\ref{thm:optimalityConditionsPv}, while the states and adjoint states satisfy the properties stated in Lemma~\ref{propo:PropertiesTwoRegions} and the refraction law shown given Theorem~\ref{thm:refractionPv} (after taking the coordinate transformation given in \eqref{eqn:coordTrans} into consideration).

On the other hand, Figure~\ref{fig:straight} depicts the setting in which  region $\P_2$ is slower than both $\P_1$ and $\P_3$. Hence, the minimum-time path is also the minimum-distance path, passing directly from $\P_1$ to $\P_3$ in a straight path, orthogonal to the boundary, as necessitated by the refraction law in Theorem~\ref{thm:refractionPv}. The optimal path shown in Figure~\ref{fig:right} is another instance in which  region $\P_2$ is a fast region. Hence, the minimum-time path takes advantage of this by traveling through $\P_2$ on the way to $\P_3$. Unlike the optimal path in Figure~\ref{fig:lanechange}, there is an asymmetry about the $x$-axis because of the heterogeneity in velocities of travel and minimum turning radii. In these two cases, the control input, the states and adjoint states also satisfy Theorem~\ref{thm:optimalityConditionsPv}, Lemma~\ref{propo:PropertiesTwoRegions}, and Theorem~\ref{thm:refractionPv}.

\begin{figure}[tp!] 
\begin{center}
\includegraphics[width=.235\textwidth,trim=15mm 2mm 10mm 2mm]{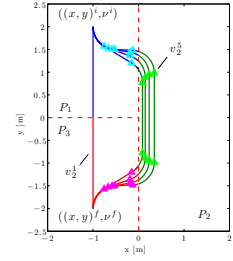}
\caption{Optimal paths from 
a given initial point 
$((x,y)^i,\nu^i) = ((-1,2),(v_1,\angle \frac{3\pi}{2}))$
to a final point 
$((x,y)^f,\nu^f) = ((-1,-2),(v_3,\angle \frac{3\pi}{2}))$
for constant velocities given by $v_1=v_3=0.25$ and minimum turning radii $r_1=r_3=0.5$ in regions $\P_1$ and $\P_3$, respectively, as well as different values of the velocity $v_2=\{0.475,0.48,0.75,1.5,5\}$ (from left to right) and a constant minimum turning radius $r_2=0.5$ in region $\P_2$. The symbols $\triangle$ marks the junctions of the bang-singular-bang controls, i.e. the locations where the path type changes. \tiny{ Minimum times: \{16.00s, 15.96s, 14.19s,  12.37s, 10.95s\}; Values of adjoints: $\overline{\adjx}=\{0.00, 2.08, 1.33, 0.67, 0.20\}$, $\overline{\adjy}_1=\{4.00,\allowbreak -3.41,-3.77,-3.94,-4.00\}$, 
$\overline{\adjy}_2=$\{4.00, -2.82$\times10^{-5}$, 3.59$\times10^{-5}$, -4.03, -9.92$\times10^{-6}$\}, $\overline{\adjy}_3=$\{4.00, 3.41, 3.77, 3.94, 4.00\}, $\overline{\adjang}_{min}=$\{0.00, -0.18,-0.20, -0.16, -0.07\}, $\overline{\adjang}_{max}=$\{0.00, 0.96, 1.33,1.67,1.90\}}}
\label{fig:sim3Reg}
  \end{center}
\end{figure}

It is also interesting to note that a phase transition can be observed with varying velocities of the regions. Figure~\ref{fig:sim3Reg} shows optimal paths from a given initial point 
$((x,y)^i,\nu^i) = ((-1,2),(v_1,\angle \frac{3\pi}{2}))$ to a final point $((x,y)^f,\nu^f) = ((-1,-2),(v_3,\angle \frac{3\pi}{2}))$
for constant velocities given by $v_1=v_3=0.25$ and minimum turning radii $r_1=r_3=0.5$ in regions $\P_1$ and $\P_3$, respectively. In region $\P_2$, the minimum turning radius is $r_2=0.5$, while the velocity is decreased from $v_2=5$ to zero. As expected, one can observe from Figure~\ref{fig:sim3Reg} (viewed from right to left) that the smaller the velocity of region $\P_2$ is, the shorter the path the particle traverses in that region. In fact, a phase transition occurs when $v_2 \approx 0.476$. 
For velocities below this value, the particle travels from $\P_1$ to $\P_3$ without going through $\P_2$.

\section{Conclusion}
\label{sec:Conclusion}

We have derived necessary conditions for the optimality of paths
with velocity and minimum turning constraints across $N$ regions. 
To establish our results, we formulated
the problem as a hybrid optimal control problem and used 
optimality principles from the literature. 
Our results provide verifiable
conditions for optimality of paths. These include conditions 
both in the interior of the regions and at their common boundary, 
as well as a refraction law for the angles which generalizes
Snell's law of refraction in optics to the current setting.
Applications of our results include optimal motion planning tasks 
for autonomous vehicles with Dubins vehicle dynamics.
By means of numerical examples, we 
verified the claims in this paper and illustrated the influence of the heterogeneous nature of the regions in state space on the resulting minimum-time paths of bounded curvature.

\section{Acknowledgments}

This research has been partially supported by ARO through
grant W911NF-07-1-0499 and MURI grant W911NF-11-1-0046, by NSF through grants 0715025 and 
CAREER Grant ECS-1150306, and by AFOSR through grant FA9550-12-1-0366. 
The simulations in this research have been carried out using the optimal control software package GPOPS 
under a public license, which requires the citation of the following references:

{\tiny \parskip 2.5pt
[$\star$1]\ 
A.~V. Rao, D.~A. Benson, C.~Darby, M.~A. Patterson, C.~Francolin, I.~Sanders,
  and G.~T. Huntington.
Algorithm 902: {GPOPS}, {A} {MATLAB} software for solving
  multiple-phase optimal control problems using the gauss pseudospectral
  method.
{\em ACM Trans. Math. Softw.}, 37(2):22:1--22:39, April 2010.

[$\star$2]\ 
D. Garg, M. ~A. Patterson, W.~ W. Hager, A. ~V. Rao, D.~ A. Benson, and G.~ T. Huntington. A unified framework for the numerical solution of optimal control problems using pseudospectral methods. {\em Automatica}, pages 1843--1851, 2010.

[$\star$3]\ 
D. Garg, M. ~A. Patterson, C. Francolin, C.~ L. Darby, G.~ T. Huntington, W.~ W. Hager, and A.~ V. Rao. Direct trajectory optimization and costate estimation of finite-horizon and infinite-horizon optimal control problems using a radau pseudospectral method. {\em Computational Optimization and Applications}, 49(2):335--358, June 2011.

[$\star$4]\ 
C.~ L. Darby, W. ~W. Hager, and A. ~V. Rao. Direct trajectory optimization using a variable low-order adaptive pseudospectral method. {\em Journal of Spacecraft and Rockets}, 48(3):433--445, 2011.

[$\star$5]\ 
C. ~L. Darby, W. ~W. Hager, and A.~ V. Rao. An hp-adaptive pseudospectral method for solving optimal control problems. {\em Optimal Control Applications and Methods}, 49(2):476--502, 2011.

[$\star$6]\ 
D. Garg. {\em Advances in Global Pseudospectral Methods for Optimal Control}. PhD thesis, University of Florida, Department of
Mechanical and Aerospace Engineering, 2011.

[$\star$7]\ 
C.~L. Darby. {\em hp-Pseudospectral Method for Solving Continuous-Time Nonlinear Optimal Control Problems}. PhD thesis, University of Florida, Department of Mechanical and Aerospace Engineering, 2011.

[$\star$8]\ 
D.~A. Benson. {\em A Gauss Pseudospectral Transcription for Optimal Control}. PhD thesis, Massachusetts Institute of Technology, Department of Aeronautics and Astronautics, 2004.

[$\star$9]\ 
G.~T. Huntington. {\em Advancement and Analysis of a Gauss Pseudospectral Transcription for Optimal Control}. PhD thesis, Massachusetts Institute of Technology, Department of Aeronautics and Astronautics, 2007.

[$\star$10]\ 
D. ~A. Benson, G.~ T. Huntington, T.~ P. Thorvaldsen, and A.~ V. Rao. Direct trajectory optimization and costate estimation via an orthogonal collocation method. {\em Journal of Guidance, Control and Dynamics}, 29(6):1435--1440, 2006.

}

\vspace{-0.2cm}

\bibliographystyle{plain}
\bibliography{long,Biblio,RGS}

\end{document}